\documentclass{amsart}
\usepackage[title]{appendix}
\usepackage{geometry}
\usepackage{url} 
\usepackage{hyperref}
\usepackage[noabbrev,capitalise]{cleveref}
\usepackage[utf8]{inputenc}
\usepackage[T1]{fontenc}
\usepackage{amsmath,amssymb,amsthm,amsfonts,mathrsfs,mathtools,relsize,comment}
\usepackage{xcolor}
\usepackage{thm-restate}

\usepackage{dsfont} 
\usepackage{mathrsfs}

\newtheorem{claim}{Claim}

\newtheorem{thm}{Theorem}[section]

\newtheorem{cor}[thm]{Corollary}

\newtheorem*{theorem*}{Theorem}

\theoremstyle{remark}

\theoremstyle{definition}
\newtheorem{definition}[thm]{Definition}
\newtheorem*{definition*}{Definition}

\usepackage{float}

\newtheoremstyle{named}{}{}{\itshape}{}{\bfseries}{.}{.5em}{\thmnote{#3 }#1}
\theoremstyle{named}

\usepackage{caption}
\usepackage[labelformat=simple]{subcaption}

\usepackage{tikz}
\usetikzlibrary{arrows}

\makeatletter
\@namedef{subjclassname@2020}{%
\textup{2020} Mathematics Subject Classification}
\makeatother

\subjclass[2020]{Primary 05C57; Secondary 05C83, 05C10}
\keywords{Cops and robbers, Graph minors, Excluded minors, Linklessly embeddable graphs}

\usepackage[shortlabels]{enumitem}

\newcommand{\N}{\mathbb{N}}

\DeclareMathOperator{\diam}{diam}

\begin{document}

\title{Improved bounds on the cop number when forbidding a minor}

\author{Franklin Kenter}
\address{Mathematics Department, United States Naval Academy, Annapolis, U.S.A.}
\email{kenter@usna.edu}
\urladdr{www.usna.edu/Users/math/kenter/}

\author{Erin Meger}
\address{School of Computing, Queen's University, Kingston, Canada}
\email{erin.meger@queensu.ca}
\urladdr{www.erinmegermath.com}

\author{J\'er\'emie Turcotte}
\address{Department of Mathematics and Statistics, McGill University, Montr\'eal, Canada}
\email{mail@jeremieturcotte.com}
\urladdr{www.jeremieturcotte.com}

\thanks{The first author is supported by the National Science Foundation (NSF) and the Office of Naval Research (ONR). The third author are supported by the Natural Sciences and Engineering Research Council of Canada (NSERC). Le troisième auteur est support\'e par le Conseil de recherches en sciences naturelles et en g\'enie du Canada (CRSNG)}

\begin{abstract}
    Andreae (1986) proved that the cop number of connected $H$-minor-free graphs is bounded for every graph $H$. In particular, the cop number is at most $|E(H-h)|$ if $H-h$ contains no isolated vertex, where $h\in V(H)$. The main result of this paper is an improvement on this bound, which is most significant when $H$ is small or sparse, for instance when $H-h$ can be obtained from another graph by multiple edge subdivisions. Some consequences of this result are improvements on the upper bound for the cop number of $K_{3,t}$-minor-free graphs, $K_{2,t}$-minor-free graphs and linklessly embeddable graphs.
\end{abstract}

\maketitle

\section{Introduction}

The game of cops and robbers is a combinatorial game played on a graph $G$ \cite{nowakowski_vertex--vertex_1983,quilliot_these_1978,aigner_game_1984}. One player plays as the cops and the other plays as the robber. The objective for the cops is the capture the robber by occupying the same vertex as the robber with one or more cops; the objective for the robber is to evade capture forever.  At the start of the game, the cops player places $m$ cops on (not necessarily distinct) vertices of the graph, then the other player places the robber on a vertex. Thereafter, starting with the cops player, the players alternate moving any of their pieces (cops or robber) to an adjacent vertex. While there are many variants of the game, we focus on the classical version of the game where a player may decline to move any (or all) of their pieces on their turn. For a graph $G$, the {\it cop number}, denoted by $c(G)$, is the minimum number of cops sufficient for the cops player to have a winning strategy \cite{aigner_game_1984}.

This game has a rich history in topological graph theory as surveyed in \cite{bonato_topological_2017}.  A classical result by Aigner and Fromme is that for any connected planar graph $G$, the cop number is at most 3 \cite{aigner_game_1984}. A long series of results have also established a relationship between the genus of the graph, $g$, and the cop number. Quilliot's \cite{quilliot_short_1985} bound is that $c(G) \le 2g + 3$; this has since been improved to $c(G) \le \lfloor \frac32 g + 3 \rfloor$ by Schr{\"o}der \cite{schroder_copnumber_2001} and subsequently to $c(G) \le \frac43g  + \frac{10}3$ by Bowler, Erde, Lehner and Pitz \cite{bowler_bounding_2021}. Lehner \cite{lehner_cop_2021} very recently showed that for any connected toroidal graph $c(G) \le 3$, solving a question of Andreae \cite{andreae_pursuit_1986}. Originally, Schr{\"o}der \cite{schroder_copnumber_2001} conjectured that $c(G) \le g + 3$; however, more daringly, Bonato and Mohar \cite{bonato_topological_2017} conjectured that, in fact, $c(G) \le g^{1/2 + o_g(1)}$.

An \emph{edge contraction} is an operation by which we obtain from $G$ a new graph $G'$ by identifying the two end vertices of an edge and removing resulting loops and multiple edges. We say $H$ is a \emph{minor} of $G$ if a graph isomorphic to $H$ can be obtained from $G$ by removing vertices and edges and by contracting edges. Given a family of graphs $\mathcal H=\{H_i\}_{i\in I}$, we say $G$ is \emph{$\mathcal H$-minor-free} if $H_i$ is not a minor of $G$ for every $i\in I$. If $\mathcal H=\{H\}$, we simply write that $G$ is $H$-minor-free. Many topological classes of graphs can be defined using forbidden minors. Most notably, Wagner \cite{wagner_uber_1937} proved that planar graphs are exactly the $\{K_5,K_{3,3}\}$-minor-free-graphs (where $K_t$ is the complete graph on $t$ vertices and $K_{s,t}$ is the complete bipartite graph with parts of size respectively $s$ and $t$). More generally, Robertson and Seymour \cite{robertson_graph_2004} famously proved that for any minor-closed family of graphs $\mathcal F$ there exists a finite set of graphs $\mathcal H$ such that $\mathcal F$ are exactly the $\mathcal H$-minor-free graphs.

Shortly after Aigner and Fromme's result on the cop number of planar graphs, Andreae \cite{andreae_pursuit_1986} studied the cop number of graphs with a forbidden minor. In particular, Andreae proved the following theorem.

\begin{restatable}{thm}{thmandreaemain}\label{thm:andreaemain}\cite{andreae_pursuit_1986}
    Let $H$ be a graph and $h\in V(H)$ be a vertex such that $H-h$ has no isolated vertex. If $G$ is a connected $H$-minor-free graph, then $c(G) \leq |E(H-h)|$.
\end{restatable}

Andreae also proves, using similar but more specific methods, that connected $K_{3,3}$-minor-free graphs and $K_5$-minor-free graphs have cop number at most 3, strengthening Aigner and Fromme's result, as well as showing that connected $K_{3,3}^-$-minor-free graphs and $K_5^-$-minor-free graphs (here $H^-$ designates the graph $H$ with one edge removed) have cop number at most 2, and that when forbidding the $(t+1)$-vertex wheel graph $W_t$ as a minor the cop number is at most $\left\lceil\frac{t}{3}\right\rceil+1$.

Joret, Kaminsky and Theis \cite{joret_cops_2010}, inspired by Andreae's paper, have considered the cop number when forbidding a subgraph or an induced subgraph, and when bounding the treewidth of the graph. The question of bounding the cop number of graphs with one or multiple forbidden induced subgraphs has gained traction is recent years, see for instance \cite{sivaraman_application_2019,chudnovsky_cops_2024,masjoody_cops_2020} and references therein. However, there have been no improvements to the upper bounds on the cop number of graphs with an excluded minor since Andreae's paper.

Our main result generalizes and improves the bounds in \cref{thm:andreaemain}. As it requires some technical definitions, we postpone its statement to \cref{sec:main}. The improved bounds are most significant when $H$ is sparse, for instance graphs when $H-h$ is a subdivision of a much smaller graph, or is small. In particular, we show a few notable applications of our main result:

\begin{itemize}
    \item We show that the cop number of linklessly embeddable graphs is at most 6; previously the upper bound was 9.
    \item We show that our main result encompasses \cref{thm:andreaemain}, as well as many of Andreae's more specific results mentioned earlier.
    \item We improve the known upper bounds on the cop number of $K_{3,t}$-minor-free graphs and $K_{2,t}$-minor-free graphs by a factor of 2.
    \item We provide an example where our method improves the cop number by a factor of $4$.
\end{itemize}

In \cref{sec:notation}, we define the notation we will be using throughout this paper. In \cref{sec:paths}, we recall the classical path guarding strategy for cops and deduce a more convenient form for our use. In \cref{sec:main}, we will state and prove our main result. Finally, in \cref{sec:applications}, we derive the corollaries of our main result mentioned above.

\section{Notation}\label{sec:notation}

We begin with some notation, which is mostly standard. For $n\in \N$, we write $[n]=\{1,\dots,n\}$. If $A$ is a set, we will use the notation $\binom{A}{2}=\left\{\{u,v\}:u,v\in A, u\neq v\right\}$ for the set of unordered pairs of distinct elements of $A$; in general we will write $uv$ or $vu$ to represent the pair $\{u,v\}$. If $f:X\rightarrow Y$ and $A\subseteq X$, then $f(A)=\{f(a):a\in A\}$ is the image of $A$. If $B\subseteq Y$, then $f^{-1}(B)=\{a\in A : f(a)\in B\}$ is the pre-image of $B$.

Let $G$ be a graph, which we always consider to be simple and finite. We denote by $V(G)$ the set of vertices of $G$ and by $E(G)\subseteq \binom{V(G)}{2}$ the set of edges of $G$. If $v\in V(G)$, we write $N(v)$ for the \emph{neighbourhood} of $v$ and $N[v]=N(v)\cup\{v\}$ for the \emph{closed neighbourhood} of $v$. Given a set $S\subseteq V(G)$, we write $N[S]=\bigcup_{v\in S}N[v]$ for the \emph{closed neighbourhood} of $S$, and we define the \emph{coboundary} of $S$ as the set $N(S)=N[S]\setminus S$, i.e. the vertices adjacent to but not in $S$ (note that $N(S)$ is not $\bigcup_{v\in S}N(v)$, as it is often defined).

If $S\subseteq V(G)$, we write $G[S]$ for the subgraph of $G$ induced by $S$. We also write $G-S$ for $G[V(G)\setminus S]$. If $x\in V(G)$, we use the shorthand $G-x$ for $G-\{x\}$. If $A\subseteq \binom{V(G)}{2}$, then we write $G-A$ for the graph on vertex set $V(G)$ with edge set $E(G)\setminus A$. If $e\in E(G)$, then we write $G-e$ for $G-\{e\}$.

A \emph{matching} in $G$ is a set of edges which pairwise do not share vertices.

We denote by $U(G)$ the graph obtained from $G$ by adding an universal vertex (a vertex adjacent to all over vertices). Given graphs $G_1,G_2$, we write $G_1+G_2$ for the graph obtained as the disjoint union of $G_1$ and $G_2$. Analogously, $mG$ is the graph obtained as the disjoint union of $m$ copies of $G$.

We use the convention that the length of a path or a cycle is the number of edges it contains. For non-empty paths this is the number of vertices minus 1; for cycles this is exactly the number of vertices. A path may have length 0, whereas a cycle necessarily has length at least 3. When a $u-v$ path has length 1, we will often write it simply as the edge $uv$ to simplify notation. The \emph{end vertices} of a path are its first and last vertices (given some arbitrary orientation of the path). In general, our cycles will have a clearly defined \emph{root} vertex; we will say the \emph{end vertices} of the cycle to be this root. We say the \emph{interior} of a path or a cycle is its set of vertices except for the end vertices. Given a non-empty $x-y$ path (or $x$-rooted cycle, if $x=y$) $P$ and $z\in V(P)$, write $P[x,z]$ for the subpath of $P$ with ends $x$ and $z$, and $P[z,y]$ analogously. If $P_1$ and $P_2$ are  internally-disjoint, respectively $x-y$ and $y-z$ paths, then we write $P_1\oplus P_2$ for their concatenation; this is also a path except when $x=z$ in which case it is an $x$-rooted cycle. For convenience, we write $\emptyset$ for the empty path.

Given an $x-y$ path $P$, we say $\{X_i\}_{i\in I}$, a family of subsets of vertices of $P$, is \emph{non-intertwined} if there do not exist distinct $i,j\in I$ and three distinct vertices $a_1,a_2\in X_i$, $b\in X_j$ such that these vertices appear in the order $a_1-b-a_2$ in $P$ (perhaps with additional vertices in between). If $P$ is an $x$-rooted cycle such that the neighbours of $x$ in the cycle are $y,z$, then we say $\{X_i\}_{i\in I}$, a family of subsets of vertices of $P$, is non-intertwined if it is is a non-intertwined family of subsets of the path $P-xy$ or of the path $P-xz$.

\section{Guarding paths}\label{sec:paths}

Let $G$ be a connected graph, and $\mathcal C$ be the set of cops playing the game of cops and robbers. We say a cop $C\in \mathcal C$ \emph{guards} a subset $S\subseteq V(G)$ if its strategy guarantees that if the robber enters $S$ it is immediately captured by $C$. For now, we require that this strategy be independent of the strategy of the other cops, that is the strategy still works even if the other cops change their strategies. Let us note that if $C$ guards $S$, then it also guards any $S'\subseteq S$; we are not claiming that $S$ is an exhaustive list of all vertices this cop is blocking the robber from entering. In general, if one vertex is guarded by multiple cops, we will select one of the cops to guard this vertex. In particular, if all the vertices guarded by $C$ are also guarded by other cops, this cop can be given a new strategy as it is currently useless.

The following result of Aigner and Fromme, originally used to prove that planar graphs have cop number at most 3, is one of the main tools in the study of the game of cops and robbers.
\begin{thm}\cite{aigner_game_1984}\label{thm:aignerpath}
	If $G$ is a connected graph, $u,v\in V(G)$ and $P$ is a shortest $u-v$ path, then there exists a strategy for one cop to, after a finite number of turns, guard $P$.
\end{thm}

Andreae noticed that the proof of \cref{thm:aignerpath} gives us the following stronger result, which we have reformulated for convenience.

\begin{thm}\cite{andreae_pursuit_1986}\label{thm:andreaepath}
	If $G$ is a connected graph, $u,v\in V(G)$, $P$ is a shortest $u-v$ path and $C$ is a cop currently on $u$, then there exists a strategy for $C$ to keep guarding $u$ and, after a finite number of turns, also guard $P$.
\end{thm}

\cref{thm:andreaepath} is indeed stronger than \cref{thm:aignerpath}, since if a cop $C$ has no current strategy, then one can still use $C$ to guard a new path by first sending the cop to $u$ in a finite number of turns and then applying the strategy given by \cref{thm:andreaepath}. The importance of this result is that sometimes a cop is already busy guarding a vertex and it is important that it keeps guarding it whilst it prepares to guard the path.

Andreae's proof of \cref{thm:andreaemain} uses this path guarding strategy repeatedly to incrementally reduce the \emph{robber's territory}, i.e. the vertices the robber can reach without being caught by a cop, while gradually constructing an instance of the forbidden minor as the game is being played. The fact that the minor is forbidden implies that the game will eventually end, with the robber being caught. This is also our approach. With the objective of being as rigorous as possible, we will use the following corollary, which is implicit in \cite{andreae_pursuit_1986} (see, in particular, Section 2). This formulation will clarify the dependencies between the strategies of the cops, since \cref{thm:andreaepath} is usually applied not to $G$ but to a subgraph of $G$ (roughly speaking, to the subgraph induced by the current robber's territory). The strategy we get thus only holds as long as the robber is guaranteed to not leave this subgraph.

\begin{cor}\label{cor:newguardedpath}
    Let $G$ be a connected graph, let $R\subseteq V(G)$ such that $G[R]$ is connected and the robber is in $R$, and let $u\in N[R]$ and $v\in N(R)$ be distinct vertices such that $v$ has at least one neighbour in $R\setminus\{u\}$.
    
    If a cop $C$ is currently on $u$, then there exists a $u-v$ path $P$ of length at least two with all internal vertices of $P$ in $R$ and a strategy for $C$ to keep guarding $u$ and, after a finite number of turns, also guard $P$, under the conditions that the robber never moves to $N(R)\setminus \{u,v\}$, and that the robber does not move to $v$ before $C$ guards $P$.
\end{cor}
\begin{proof}
    Consider the graph $G'=G[R\cup\{u,v\}]-uv$. Let $P$ be a shortest $u-v$ path in $G'$; such a path exists given that $v$ has a least one neighbour in $R$ which is not $u$. By choice of $G'$, the only vertices of $P$ which are possibly not in $R$ are $u,v$, and so all internal vertices of $P$ are in $R$. Given that $uv\notin E(G')$, $P$ must contain at least one vertex other than $u,v$, and so $P$ must have length at least two. 
    
    Apply \cref{thm:andreaepath} to get a strategy for $C$ to guard $u$ and, after a finite number of turns, also guard $P$, if the game is played on $G'$. We claim this strategy is also a valid strategy when playing on $G$. Of course, given that $G'$ is a subgraph of $G$, every move of the cop remains valid. We need to show that the robber cannot, without being caught, make any move on $G$ which it could not have done on $G'$. Suppose to the contrary that we are at the first turn where the robber makes such a move. The first type of illegal move is using the edge $uv$, which implies the robber was on $u$ or $v$ at the previous turn. The first case is impossible, given that $u$ was guarded by $C$, and the latter is impossible because either the robber was supposed to not move to $v$ (if the cop was not yet guarding $P$) or $C$ was guarding $v$ (if the cop was guarding $P$). The other type of illegal move is moving outside $R\cup\{u,v\}$. By hypothesis, the robber can only leave $R$ through either $u$ or $v$, which is impossible by the same argument as above.
\end{proof}

\section{Main result}\label{sec:main}

In this section, we state and prove the main result of this paper, which is an upper bound on the cop number for graphs $G$ with some forbidden minor $H$. By and large, the proof optimizes and greatly extends the techniques used in the proof of \cref{thm:andreaemain}, with very technical modifications. We summarize these key elements following the proof.

The statement of the result requires the two following definitions. 

\begin{definition}\label{def:path-decomp}
Given a graph $H$, we say that the tuple $\mathcal{H} = (h, W, \mathcal{P}, M, f)$ is a \emph{decomposition} of $H$, where
\begin{enumerate}[label=(\alph*)]
    \item $h \in V(H)$,
    \item\label{item:core} $\emptyset\neq W\subseteq V(H-h)$,
    \item\label{item:paths} $\mathcal P$ is a collection of distinct pairwise internally vertex-disjoint paths and (rooted) cycles with end vertices in $W$ such that every edge of $H-h$ is contained in some $P\in\mathcal P$, 
    \item\label{item:defmatching} $M\subseteq \mathcal P$ is a collection of paths of length 1 which forms a matching of $H-h$, and
    \item\label{item:mappingends} $f:W\rightarrow \mathcal P\setminus M$ is such that $u$ is an end of $f(u)$ for every $u\in W$. 
\end{enumerate}
\end{definition}

\cref{subfig:decomposition} gives an example of a decomposition.

Conditions \ref{item:core} and \ref{item:paths} can be seen as stating that the graph $H-h$ is a subdivision of a multigraph (with loops authorized) on the vertex set $W$. Note that in condition \ref{item:defmatching} $M$ does not need to be a perfect matching, and may be empty. Given that a path or (rooted) cycle $P\in \mathcal P$ may only have at most 2 vertices, condition \ref{item:mappingends} implies that $|f^{-1}(P)|\in\{0,1,2\}$.

Intuitively, a decomposition of $H$ is, after choosing some vertex $h\in H$, a way of representing $H-h$ around a ``core'' set of vertices $W$, between which there are paths (those in $\mathcal P$). We will use this decomposition as the blueprint when we attempt to construct $H$ as a minor of the graph $G$ on which the game is played.

To further motivate this definition, we broadly outline the idea behind the proof. We will progressively construct a minor of $H$ inside $G$, using the properties of the game to show that we can add every vertex and edge of $H$ to our partial minor. The robber's territory, that is the region in which the robber will be confined to, will correspond (will be contracted to) to the vertex $h$. The cops' territory, that is the region guarded by the cops, will consist of bags (denoted $A_w$, for every $w\in W$) and paths between these bags (denoted $Q_P$, for $P\in \mathcal P$), with the property that if $P$ is a path in $H$ between $w_1$ and $w_2$, then $Q_P$ will be a path in $G$ between $A_{w_1}$ and $A_{w_2}$. If we can ensure that every $A_w$ is non-empty (and has a neighbour in the robber's territory) and that every $P\in \mathcal P$ has a corresponding (and sufficiently long) path $Q_P$ in $G$, we will have obtained a minor of $H$ in $G$. Broadly speaking, the paths $Q_P$ in the cop territory will completely contain every vertex which is outside of but adjacent to the robber's territory, and a cop will always guard every such vertex, ensuring that the robber is confined to its territory. Indeed, for every $P\in \mathcal P$, a group of cops $\mathcal C_P$ will be assigned to guard the path $Q_P$. If, for example, a path $P$ between $w_1$ and $w_2$ does not yet have a corresponding path $Q_P$ in our model, one of those cops will, using the results of the previous section, start protecting a path path between $A_{w_1}$ and $A_{w_2}$ going through the robber's territory, which will thus reduce it. In some specific cases, we will be able to add a path to the model without requiring any cops to protect it: these are edges of $M$. The ends of this path will be guarded by cops assigned to other paths of the model; this is why we require $M$ to be a matching. Furthermore, as we have noted above, we want every bag $A_w$ to have a least one neighbour in the robber's territory. In our proof, we will in fact be able to guarantee that only one vertex of $A_w$ will be adjacent to the robber's territory. As this vertex only needs to be guarded by one cop, the role of the function $f$ is to indicate the group of cops (the group assigned to $Q_{f(w)}$) which will be responsible for guarding this vertex. We will formally define this partially constructed minor in the context of the game as a \emph{state}, an example of which is shown in \cref{fig:state} below.
 
We may then define the following parameter for each path of $\mathcal P$. It will always be approximately be the length of the path, but takes into account these technicalities: we only need to know the length of the part of the path for which the corresponding vertices in $G$ needs to be guarded by cops.

\begin{definition}\label{def:length-param}
    Given a decomposition $\mathcal H=(h,W,\mathcal P,M,f)$ of a graph $H$, we define for each path $P \in \mathcal{P}$ the following parameter:
    	$$\ell_P=\begin{cases}
		0&P\in M\\
		\max(|E(P)|-1+|f^{-1}(P)|,1)&P\notin M.
	\end{cases}$$
\end{definition}

We are now ready to state our main result.

\begin{thm}\label{thm:main}
    If $\mathcal{H}$ is a decomposition of a graph $H$ and $G$ is a connected $H$-minor-free graph, then 
    $$c(G)\leq \mathds{1}_{\ell}+\sum_{P\in \mathcal P}\left\lceil\frac{\ell_P}{3}\right\rceil,$$
    where the indicator function $\mathds{1}_{\ell}$ is equal to $1$ if and only if there is some $P \in \mathcal{P}$ with $\ell_P \notin \{0,1,2,4\}$.
\end{thm}

In this theorem, we do not impose any conditions on which decomposition is picked. As can be seen in \cref{def:path-decomp}, a graph $H$ may have multiple possible decompositions. When using this theorem, the best bound will be obtained by choosing an optimal decomposition of $H$, roughly speaking by choosing a decomposition that yields the smallest possible sum of $\ell_P$. Note that the minor relation is transitive, and so if $H$ is a minor of $H'$, then if a graph is $H$-minor-free, it is also $H'$-minor-free. Hence, one might obtain a better upper bound by applying \cref{thm:main} to $H'$ instead of $H$. We make such an application of \cref{thm:main} in the proof of \cref{cor:k2t}. This is also useful when $H$ does not have any decomposition, for instance if $H-h$ contains an isolated vertex. Further discussion on applications of this result is provided in \cref{sec:applications}.  \\

The structure of the proof is as follows: 
\begin{itemize}
    \item Set up terminology surrounding both $G$ and the $\mathcal{H}$ decomposition, using precisely the number of cops given by the upper bound.
    \item Define a game state to detail the particular relationship between $G$ and the forbidden minor $H$ through its decomposition. This state relates paths in $G$ to paths in $H$, using the terminology of \emph{initialized} to indicate that the feature will later be used to build the minor. Furthermore, we will say that the cops are \emph{active} when they have been assigned a particular strategy and are actively guarding vertices of $G$. 
    \item Define a partial order on the game states. The robber will be captured when the robber's territory decreases to zero, which can be achieved by taking smaller and smaller game states.
    \item Assume that we are in some minimal game state, where the robber's territory is non-zero for the sake of contradiction. That is, we assume that our current bound is insufficient for the cops to win.
    \item Explore the game state to show that certain features of the decomposition must be present in $G$, by assuming their absence and finding a smaller game state, thus contradicting the minimality of the game state.
    \item Finally, we will show that with all features present, we must in fact have an $H$ minor of $G$, which is forbidden, contradicting the assumption that the game is a minimal game state with non-empty robber's territory.
\end{itemize}
We now prove \cref{thm:main}.

\begin{proof}
    Let $G$ be a connected $H$-minor-free graph, and $\mathcal{H} = (h, W, \mathcal{P}, M, f)$ be a decomposition of $H$. We will play the game of cops and robbers with $\mathds{1}_{\ell}+\sum_{P\in \mathcal P}\left\lceil\frac{\ell_P}{3}\right\rceil$ cops. Let $\mathcal{C}$ be the set of all cops.
 
    For each $P\in \mathcal P$, we define $\mathcal C_P$ to be a set of $\left\lceil\frac{\ell_P}{3}\right\rceil$ cops, such that for all pairs of distinct paths $P_1, P_2 \in \mathcal{P}$, $\mathcal C_{P_1} \cap \mathcal C_{P_2} = \emptyset$. Since $\ell_P = 0$ if and only if $P \in M$, it follows that $\mathcal C_P=\emptyset$ if and only if $P\in M$. In particular, for every vertex $w\in W$, $\mathcal C_{f(w)}\neq\emptyset$, since the function $f$ maps only to non-matching paths in $\mathcal{P}$.  If the indicator function $\mathds{1}_{\ell}=1$ (if some $P\in \mathcal P$ is such that $\ell_P\notin\{0,1,2,4\}$), we then define $C_\ell$ be an additional, distinct cop. Note that the set $\mathcal{C}_P$ to which each cop belongs, may change throughout the proof as the cops ``switch roles'' but they will always do so in a way that leaves the sizes of these sets unchanged.

    In order to show that the cops have a winning strategy, we next need to define the concept of a game state. We will then define a partial order between game states. We will call the process of going from one game state to another smaller game state a \emph{transition}.

    \begin{definition*}
        We say the game is in \emph{state} $(\mathcal A,\mathcal Q,R,s)$ if all of the following hold.
	    \begin{enumerate}[label=(\arabic*)]
	       \item\label{item:defbags} $\mathcal A=(A_w)_{w\in W}$ is a collection of pairwise disjoint subsets of $V(G)$ (which we will call \emph{bags}) such that for every $w\in W$, $G[A_w]$ is either connected or contains no vertices. We say $w$ is \emph{initialized} if $A_w\neq \emptyset$.
            \item\label{item:defpaths} $\mathcal Q=(Q_P)_{P\in \mathcal P}$ is a collection of pairwise internally vertex-disjoint paths such that if $P$ has end vertices $u,v$, then either $Q_P$ is empty or $Q_P$ has end vertices respectively in $A_u$ and $A_v$, and such that internal vertices are not in any of the sets in $\mathcal A$. If $P$ is not a path but a rooted cycle, the end vertices of $Q_P$ are allowed (but not obliged) to be the same (i.e. $Q_P$ is allowed to be a rooted cycle). We say $P\in \mathcal P$ is \emph{initialized} if $Q_P$ is not empty.
            \item\label{item:matching} If $u,v\in W$ are initialized and $uv\in M$, then $uv$ is initialized.
            \item\label{item:Rdef} $R$ is set of vertices of the connected component of $G-\left(\bigcup_{w\in W}A_w\cup\bigcup_{P\in \mathcal P}V(Q_P)\right)$ containing the robber.
            \item\label{item:max1} For each $w \in W$, $A_w$ contains at most one vertex adjacent to $R$.
            \item\label{item:funcs} $s:\mathcal C\rightarrow 2^{V(G)\setminus R}$ is a function such that $C\in \mathcal C$ is following a strategy to guard the vertices in $s(C)$ which is irrespective of the strategy of the other cops, but holds only as long as the robber does not leave $R$ by moving to a vertex in $N(R)\setminus s(C)$. A cop is said to be \emph{active} if $s(C)\neq \emptyset$. Inactive cops $C$ may follow any strategy.
            \item\label{item:coboundary} Every vertex in the coboundary of $R$ is in one of the images of $s$, i.e. $N(R)\subseteq \bigcup_{C\in \mathcal C}s(C)$.
            \item\label{item:nonintertwined} If $P\in \mathcal P$ is initialized, then $s(\mathcal C_P)$ is a non-intertwined family of subsets of $V(Q_P)$.
            \item\label{item:noninitsubset} If $P\in \mathcal P$ is uninitialized, with end vertices $u,v$,  and $C\in \mathcal C_P$, then $s(C)$ is either empty or contains a unique vertex, in $A_u$ or $A_v$.
            \item\label{item:bagguard} If $A_w$ contains a vertex $x$ adjacent to $R$, then $x\in s(C)$ for some $C\in \mathcal C_{f(w)}$.
            \item\label{item:extracop} The extra cop $C_\ell$  is inactive, whenever it exists (ie. when $\mathds 1_\ell=1$).
	   \end{enumerate}
 
        Furthermore, we will use the notations $|\mathcal A|=|\{w\in \mathcal W:A_w\neq \emptyset\}|$ and $|\mathcal Q|=|\{P\in \mathcal P:Q_P\neq \emptyset\}|$ to denote respectively the number of initialized vertices of $W$ and the number of initialized paths of $P$.  When helpful, we will call the graph $G\left[\bigcup_{w\in W}A_w\cup \bigcup_{P\in \mathcal P}V(Q_P)\right]$ the \emph{model} since it is a partial construction of a graph which could be contracted into $H-h$.
    \end{definition*}

    A visualization of an example of a game state is provided in \cref{fig:state}.

    \begin{figure}
        \begin{subfigure}[t]{\textwidth}
         \centering
         \begin{tikzpicture}[scale=1.8,
			dot/.style = {circle, fill, minimum size=#1,
			inner sep=0pt, outer sep=0pt},
			dot/.default = 4pt]

            \node[dot] (h) at (0,-0.05) [label={[scale=1] left : $h$}] {};
            \node[dot] (a) at (-1,1) [label={[scale=1] left: $a$}] {};
            \node[dot] (b) at (0,1) [label={[scale=1] below left: $b$}] {};
            \node[dot] (c) at (1,1) [label={[scale=1] right: $c$}] {};
            \node[dot] (d) at (-0.5,2) [label={[scale=1] above: $d$}] {};
            \node[dot] (e) at (0.5,2) [label={[scale=1] above: $e$}] {};

            \node (P1) at (-0.48,0.62) [label={[scale=1] above: $P_1$}] {};
            \node (P2) at (0.48,0.62) [label={[scale=1] above: $P_2$}] {};
            \node (P3) at (-1,1.27) [label={[scale=1] above: $P_3$}] {};
            \node (P4) at (-0.4,1.25) [label={[scale=1] above: $P_4$}] {};
            \node (P5) at (0.4,1.25) [label={[scale=1] above: $P_5$}] {};
            \node (P6) at (0.96,1.27) [label={[scale=1] above: $P_6$}] {};
            \node (P7) at (0,1.92) [label={[scale=1] above: $P_7$}] {};

            \node (W) at (2,1.85) [label={[scale=1] right: $W=\{a,b,c,d,e\}$}] {};
            \node (P) at (2,1.4) [label={[scale=1] right: $\mathcal P =\{P_1,P_2,P_3,P_4,P_5,P_6,P_7\}$}] {};
            \node (M) at (2,0.95) [label={[scale=1] right: $M=\{P_4,P_6\}$}] {};

            \node (f1) at (2,0.5) [label={[scale=1] right: $f(a)=P_3, f(b)=P_1, f(c)=P_2$}] {};

            \node (f2) at (2,0.2) [label={[scale=1] right: $f(d)=P_3, f(e)=P_5$}] {};

            \node[dot] (i1) at (-0.9,1.2) {};
            \node[dot] (i2) at (-0.8,1.4) {};
            \node[dot] (i3) at (-0.7,1.6) {};
            \node[dot] (i4) at (-0.6,1.8) {};

            \draw (h) to (a);
            \draw (h) to (b);
            \draw (h) to (c);
            \draw (a) to (b);
            \draw (a) to (i1);
            \draw (i1) to (i2);
            \draw (i2) to (i3);
            \draw (i3) to (i4);
            \draw (i4) to (d);
            \draw (b) to (c);
            \draw (b) to (d);
            \draw (b) to (e);
            \draw (c) to (e);
            \draw (d) to (e);

		\end{tikzpicture}
        \caption{Decomposition $(h,W,\mathcal P,M,f)$ of a graph $H$.}
        \label{subfig:decomposition}
        \end{subfigure}
        \par\bigskip
        \begin{subfigure}[t]{\textwidth}
         \centering
         \usetikzlibrary{decorations.pathmorphing}
         \begin{tikzpicture}[scale=0.55,
			dot/.style = {circle, fill, minimum size=#1,
			inner sep=0pt, outer sep=0pt},
			dot/.default = 4pt]

            \draw  (0,-2) circle[radius=4.5] node {};
            \draw   (-8,8) circle[radius=3] node {};
            \draw   (0,8) circle[radius=3] node {};
            \draw   (8,8) circle[radius=3] node {};
            \draw   (-4,17) circle[radius=3] node {};

            \node [dot] (robber) at (-1.5,-3.3) [label={[scale=1] right: Robber}] {};

            \node (R) at (3.5,-5.5) [label={[scale=1] right: $R$}] {};
            \node (Aa) at (-8,5) [label={[scale=1] below: $A_a$}] {};
            \node (Ab) at (0,5) [label={[scale=1] below: $A_b$}] {};
            \node (Ac) at (8,5) [label={[scale=1] below: $A_c$}] {};
            \node (Ad) at (-2,19.5) [label={[scale=1] right: $A_d$}] {};

            \node (Q3) at (-8,13) [label={[scale=1] right: $Q_{P_3}$}] {};
            \node (Q4) at (-0.4,12.7) [label={[scale=1] right: $Q_{P_4}$}] {};
            \node (Q2) at (2.85,10.5) [label={[scale=1] right: $Q_{P_2}$}] {};

            \draw[decorate, decoration={snake,  amplitude=2pt, segment length=20pt}]           (-7,10.2)  -- (-5,14.7);
            \draw[decorate, decoration={snake,  amplitude=2pt, segment length=20pt}]           (2,9.5)  -- (6,9.5);
            \draw (-2.5,15)  -- (2,9.5);

            \node [dot] at (2,9.5) {};
            \node [dot] at (-5,14.7) {};
            \node [dot] at (-5.42,14) {};
            \node [dot] at (-5.9,13) {};
            \node [dot] at (-6,12.165) {};
            \node [dot] at (-6.6,10.93) {};
            \node [dot] at (4,9.5) {};
            \node [dot] at (5,9.63) {};
            \node [dot] at (6,9.5) {};

            \draw (2,9.5)  -- (0.5,0.7);
            \draw (-5,14.7)  -- (-0,0.7);
            \draw (-5.42,14)  -- (-0.5,0.7);
            \draw (-5.9,13)  -- (-1,0.7);
            \draw (-6,12.165)  -- (-1.5,0.7);
            \draw (-6.6,10.93)  -- (-2,0.7);
            \draw (4,9.5)  -- (1,0.7);
            \draw (5,9.63)  -- (1.5,0.7);
            \draw (6,9.5)  -- (2,0.7);

            \draw[rounded corners=5pt, color=blue] (2.7,10) rectangle (6.5,9) {};
            \draw[rounded corners=5pt, color=blue] (1.5,9) rectangle (2.5,10) {};
            \draw[rounded corners=5pt,rotate=-25, color=blue] (-11.2, 6) rectangle (-10.2, 8.8) {};
            \draw[rounded corners=5pt,rotate=-25, color=blue] (-11.2, 9) rectangle (-10.2, 11.5) {};

            \node (Aa) at (-9.5,12.5) [label={[scale=1, color=blue] below: $s(C_1)$, $C_1\in \mathcal C_{Q_{P_3}}$}] {};
            \node (Ab) at (-8.5,15) [label={[scale=1,color=blue] below: $s(C_2)$, $C_2\in \mathcal C_{Q_{P_3}}$}] {};
            \node (Ac) at (0,9.5) [label={[scale=1,color=blue] below: $s(C_3)$, $C_3\in \mathcal C_{Q_{P_1}}$}] {};
            \node (Ad) at (5.5,8.5) [label={[scale=1,color=blue] right: $s(C_4)$, $C_4\in \mathcal C_{Q_{P_2}}$}] {};

		\end{tikzpicture}
        \caption{State $(\mathcal A,\mathcal Q,R,s)$. We notice that $e$ is the only uninitialized vertex of $W$, and that the uninitialized paths are $P_1$, $P_5$, $P_6$ and $P_7$. As $b$ and $d$ are initialized, it was obligatory for $P_4$ to be initialized. Further notice that despite $P_1$ not being initialized, one of the cops of $\mathcal C_{Q_{P_1}}$ is active, and is protecting a vertex in $A_b$ by sitting on it. Finally, note that every vertex adjacent to $R$ that is also in one of the bags $A_w$ is protected by a cop in $\mathcal C_{Q_{f(w)}}$. As $P_4\in M$, there no cops assigned to protect it, so it cannot contain in its interior any vertices adjacent to $R$.}
        \end{subfigure}

        \caption{Example of a decomposition of a graph $H$ and of a state of a game played on an $H$-minor-free graph $G$.}
        \label{fig:state}
    \end{figure}
     
    Note that condition \ref{item:extracop} does not state that the extra cop is never used. It simply implies that when the game in a specific state, it not used. This cop will however be used when transitioning from one state to another, as we will see below.

	Let us also note that once the game is in a state, it may remain in this state as long as the cops' strategies do not change. Indeed, the robber is in $R$ by \ref{item:Rdef}, and cannot leave $R$ due to \ref{item:coboundary}, as long as the cops maintain their current strategies, which is possible, as specified in \ref{item:funcs}, as long as the robber does not leave $R$. In general, note that every cop $C\in \mathcal C$ may change its actual strategy, as long as the vertices it guards are the same, since the strategies of the other cops do not depend on it.
	
    \begin{definition*}
	   We can define a partial order on states by setting $(\mathcal A',\mathcal Q',R',s')<(\mathcal A,\mathcal Q,R,s)$ if 
	   \begin{enumerate}[label=(\roman*)]
	       \item\label{item:reduction1} $R'\subsetneq R$ (the robber's territory is decreased),
	       \item\label{item:reduction2} $R'=R$ and $\sum_{C\in \mathcal C}|s(C)|>\sum_{C\in \mathcal C}|s'(C)|$ (the number of guarded vertices, with multiplicity, decreases),
	       \item\label{item:reduction3} $R'=R$, $\sum_{C\in \mathcal C}|s(C)|=\sum_{C\in \mathcal C}|s'(C)|$ and $|\mathcal A'|+|\mathcal Q'|<|\mathcal A|+|\mathcal Q|$ (the number of pieces of the model decreases), or
	       \item\label{item:reduction4} $R'=R$, $\sum_{C\in \mathcal C}|s(C)|=\sum_{C\in \mathcal C}|s'(C)|$, $|\mathcal A'|+|\mathcal Q'|=|\mathcal A|+|\mathcal Q|$ and $\sum_{w\in W}|A_w'|>\sum_{w\in W}|A_w|$ (the total size of the bags increases).
	   \end{enumerate}

    \end{definition*}
 
	It is easy to see that this defines a well-founded relation on the set of states, in particular, using that these parameters have a finite number of possible values. 

    If the cops change strategies changes to bring the game from one state to a smaller state, we will say the \emph{type of the transition} is the condition (either \ref{item:reduction1}, \ref{item:reduction2}, \ref{item:reduction3} or \ref{item:reduction4}) in the definition of the partial order by virtue of which the new state is smaller.
	
	For brevity, in general when defining a new smaller state $(\mathcal A',\mathcal Q',R',s')$, we will only define the values $A_w'$ ($w\in W$), $Q_P'$ ($P\in \mathcal P$) and $s'(C)$ ($C\in \mathcal C$) which are different from $(\mathcal A,\mathcal Q,R,s)$ (in particular, all cops except those mentioned will maintain their current strategies). In general, we will also not explicitly define $R'$ as it will always be the component of $G-\bigcup_{w\in W}A_w'\cup \bigcup_{P\in \mathcal P}V(Q_P')$ containing the robber. In all instances, we will indeed have $R'\subseteq R$ since only vertices not adjacent to $R$ will ever be removed from the model.

    With the technical definitions completed, we now proceed with proving that the cops have a winning strategy. Suppose for the sake of contradiction that the robber has a strategy to escape any strategy employed by these cops.
    
    In order to find a contradiction, we first place the cops arbitrarily on $G$ and assign them no strategy. Then, \[\left((\emptyset)_{w\in W},(\emptyset)_{P\in \mathcal P},V(G),s(\cdot)=\emptyset\right)\] is a valid state since all of the conditions hold trivially. Then, the cops will follow a strategy to minimize the game state. Given the defined partial order, it must be the case that after some finite amount of time the game is in some minimal state $(\mathcal A,\mathcal Q,R,s)$.

    To avoid repetition, we first explain more precisely why, with the strategies we will use, \ref{item:funcs} continues to hold for every cop $C$ during a transition from $(\mathcal A,\mathcal Q,R,s)$ to $(\mathcal A',\mathcal Q',R',s')$. There are essentially two kinds of strategy changes we will use, which we summarize here:\\

    \textit{Case I:} The first kind to consider is that, as noted above, the cop $C$ maintains its current strategy (and it has maintained it during the transition between the game states) and that $s'(C)\subseteq s(C)$ (of course, we will always choose $s'(C)$ in a way that $s(C)\setminus s'(C)$ will only contain vertices which are also guarded by other cops, so \ref{item:coboundary} will still hold). To show that \ref{item:funcs} still holds, we suppose that we know that the robber is in $R'$ but never leaves $R'$ by moving to a vertex in $N(R')\setminus s'(C)\supseteq N(R')\setminus s(C)$, and we must show that $C$ guards $s'(C)$. We however know that $C$ guards $s(C)\supseteq s'(C)$ as long as the robber does not leave $R$ by moving to a vertex in $N(R)\setminus s(C)$, so it suffices to verify this last condition. Recall that $R'\subseteq R$. The robber cannot leave $R$ directly from $R'$ since it would have to be through a vertex in $N(R')\setminus s(C)$, which is forbidden. The robber also cannot leave $R$ by first going to a vertex in $R \setminus R'$, as this would be forbidden since $R\cap N(R')\subseteq N(R')\setminus s(C)$ since $s(C)$ contains no vertex of $R$. \medskip

	\textit{Case II:} The other kind uses \cref{cor:newguardedpath}: we get a new $x-y$ path (or $x$-rooted cycle if $x=y$) $Q$ with internal vertices in $R$ such that, after a finite of moves, a cop $C$ will be guarding $Q$, under the condition that the robber does not leave $R$ by going on a vertex of $N(R)\setminus \{x,y\}$. When using this argument, generally $V(Q)$ will be the new part of the model and we will set $s'(C)=V(Q)$, although it will be clear in the proof when this is not the case. To show that \ref{item:funcs} holds, we need to prove that if the robber does not leave $R'$ by moving to a vertex in $N(R')\setminus s'(C)$, the new strategy that $C$ is following still works. Note that by the definition of $R'$ in \ref{item:Rdef}, $N(R')\subseteq N(R)\cup V(Q)$. If the robber were to leave $R$ by going to a vertex in $N(R)\setminus \{x,y\}$, it could be directly from $R'$ if this vertex is in $N(R')\setminus V(Q)$, which is forbidden. Otherwise, the robber would need to go first through $V(Q)$ to reach vertices in $R\setminus R'$ which is impossible given that $V(Q)$ is guarded by $C$. \\
 
	We will now prove a series of claims about the minimal state $(\mathcal A,\mathcal Q,R,s)$ the game is currently in. Most of the proofs of these claims will be by contradiction, showing that if the claim does not hold, then the cops can, after a finite number of turns, bring the game into a smaller state.

	
	\begin{claim}\label{claim:sadjacent}
	    For every $C\in \mathcal C$, $s(C)$ contains only vertices adjacent to $R$.
	\end{claim}
	\begin{proof}
	    If for some $C\in \mathcal C$, there exists $x\in s(C)$ such that $x$ has no neighbour in $R$, then $C$ does not need to explicitly guard $x$, as the robber cannot reach that vertex (see \ref{item:coboundary}). Let $s'(C)=s(C)\setminus\{x\}$. Then, $(\mathcal A,\mathcal Q,R,s')$ is a new valid state, with transition of type \ref{item:reduction2}, which is a contradiction to the minimality of the current game state.
	\end{proof}
	
	\begin{claim}\label{claim:sdisjoint}
	    $s$ has disjoint images, i.e. if $C_1,C_2\in \mathcal C$ are distinct cops, then $s(C_1)\cap s(C_2)=\emptyset$.
	\end{claim}
	\begin{proof}
	    If $x\in A_w$ for some $w\in W$,  is in multiple elements of $s(\mathcal C)$, condition \ref{item:bagguard} implies that at least one of the cops guarding $x$ is some $C\in \mathcal C_{f(w)}$. Defining $s'$ such that $x$ is only contained in $s'(C)$, and otherwise identically to $s$, yields a new state $(\mathcal A,\mathcal Q,R,s')$ with transition, similarly to the previous claim, of type \ref{item:reduction2}.
	    
	    If a vertex $x$ in the interior of $Q_P$ (for some $P\in \mathcal P$) is in multiple elements of $s(\mathcal C)$, choose arbitrarily which of these cops will keep guarding $x$ and proceed as previously. Note that \ref{item:nonintertwined} is maintained, the only difference is that the sets of $s(\mathcal C_P)$ can no longer overlap.
	    
	    In both of these situations, we get a contradiction to the minimality of the current game state.
	\end{proof}
	
	\begin{claim}\label{claim:Pneighbours}
	    If $P\in \mathcal P\setminus M$ is initialized, either the interior of $Q_P$ contains a neighbour of $R$, or $P = f(u)=f(v)$ for distinct $u,v \in W$ so that both ends of $Q_P$, say $x \in A_u$ and $y \in A_v$ have neighbours in $R$. 
	\end{claim}
	\begin{proof}
	    Suppose there exists $P\in \mathcal P\setminus M$ for which the statement does not hold. Let $x\in A_u,y\in A_v$ be the end vertices of $Q_P$ (note that it is possible that $x=y$ if $P$ is a cycle).
	    
	    By \cref{claim:sadjacent}, only vertices adjacent to $R$ appear in the elements of $s(\mathcal C_P)$, and by \ref{item:nonintertwined} only vertices of $Q_P$ can appear in $s(\mathcal C_P)$. In particular, given that $Q_P$ contains no internal vertex adjacent to $R$, only $x,y$ can appear in $s(\mathcal C_P)$. Note that by \ref{item:bagguard} and \cref{claim:sdisjoint}, the ends $x,y$ of $Q_P$ may only appear in $s(\mathcal C_P)$ if respectively $P=f(u)$, $P=f(v)$. Since we are assuming that the claim does not hold for $P$, this holds for at most one of $x,y$ (unless $x=y$). Hence, either $s(C)=\emptyset$ for all $C\in\mathcal C_P$, or only one cop of $\mathcal C_P$ is active and only guards one vertex in $A_u$ (without loss of generality). These two cases will correspond to the two possible cases in \ref{item:noninitsubset} of the new smaller state, which we now define. Set $Q_P'=\emptyset$. Note that $R$ is still a component of $G-\left(\bigcup_{w\in W}A_w\cup\bigcup_{P\in\mathcal P}V(Q_P')\right)$, given that no neighbour of $R$ was in the interior of the removed path. It is then easy to verify that $(\mathcal A,\mathcal Q',R,s)$ is a new valid state with transition of type \ref{item:reduction3}, which is a contradiction to the minimality of the current game state.
	\end{proof}
	
	\begin{claim}\label{claim:copsitting}
	    If $C\in \mathcal C$ is such that $s(C)=\{x\}$, we may assume that $C$ is guarding $x$ by sitting on it.
	\end{claim}
	\begin{proof}
	    Suppose $C\in \mathcal C$ is such that $s(C)=\{x\}$. Of course, there are possibly many strategies that $C$ could be using to block the robber from moving to $x$. For instance, $C$ could be using some larger path guarding strategy such as in \cref{cor:newguardedpath}, or sitting on a neighbour of $x$ until the robber enters $x$. 
        
        Given that $C$ is guarding $x$ if the robber only moves in $G[R\cup \{x\}]$, the distance (at the cops' turn) between $x$ and $C$ in $G$ cannot be more than one larger than the distance between $x$ and the robber in $G[R\cup\{x\}]$. Indeed, otherwise the robber could follow a shortest path to $x$ and not be caught by $C$ on the way.
	    
	    Let $C$ abandon its current strategy and move towards $x$ via the shortest path in $G$ until the cop reaches $x$, after which it will sit on $x$ to guard it. We claim that once this is done, the state of the game will be unchanged; it suffices to show that the robber could not have escaped $R$ through $x$, given that all of other cops may follow their strategies as long as the robber not leave $R$, as specified in \ref{item:funcs}. Given the distances between $x$ and $C$ and the robber discussed above, the cop will either arrive at $x$ before the robber does or capture the robber on $x$.
	    
	    Given there are a finite number of cops and this strategy takes at most $\diam(G)\leq |V(G)|$ turns, we can apply the above strategy for every cop if needed. Hence, from now on, if a cop is only guarding one vertex, we may suppose it is sitting on that vertex.
	    
	    Note that once the change of strategy is complete, the state of the game is unchanged.
	\end{proof}
	
    \begin{claim}\label{claim:initializedvertex}
        If $w\in W$ is initialized, then $A_w$ contains a vertex adjacent to $R$.
    \end{claim}
	\begin{proof}
	    Suppose to the contrary that there is some $w\in W$ which is initialized but such that $A_w$ contains no vertex adjacent to $R$. There are a few cases to consider here.
	
	    If there are no initialized paths in $\mathcal P\setminus M$ incident to $w$, then set $A_w'=\emptyset$. If $wv\in M$ for some $v\in W$, let $Q_{wv}'=\emptyset$. The cop assignment $s$ is still well defined. Indeed, $A_w$ contained no vertex adjacent to $R$, and thus by \cref{claim:sadjacent} none of its elements was guarded by a cop, and in the case with $wv$ is in the matching, no internal vertex of $Q_{wv}$ is guarded by a cop given that $\mathcal C_{wv}=\emptyset$. It is easy to see that $(\mathcal A',\mathcal Q',R,s)$ defines a new valid state, this time with transition of type \ref{item:reduction3}. \medskip
	
	    Suppose now $f(w)$ is initialized. Say $f(w)$ has ends $w,v$ and $Q_{f(w)}$ has end vertices $x,y$, with $x\in A_w$ and $y\in A_v$ (note that if $w=v$, it is possible that $x=y$). We know that $x$ is not adjacent to $R$. By \cref{claim:Pneighbours}, $Q_{f(w)}$ must then contain a vertex in its interior which is adjacent to $R$. Take $z$ such a vertex which is as close as possible to $x$ in $Q_{f(w)}$ (when traversing it from $x$ to $y$). Let $A_w'=A_w\cup V(Q_{f(w)}[x,z])$ and $Q_{f(w)}'=Q_{f(w)}[z,y]$ (note that $Q_{f(w)}'$ is necessarily a path, even if $Q_{f(w)}$ was a cycle). As $A_w$ contained no vertex adjacent to $R$ and by our choice of $z$, $A_w'$ still respect \ref{item:max1}. Then, $(\mathcal A',\mathcal Q',R,s)$ will define a new valid state, here the transition being of type \ref{item:reduction4}, since the only change is that $A_w'$ absorbed some vertices of $Q_{f(w)}$. Note that it is important for \ref{item:bagguard} that we absorbed parts of $Q_{f(w)}$ and not of any incident path, so that the vertex of $A_w'$ be guarded by one of the cops of $\mathcal C_{f(w)}$. \medskip
	
	    Hence, we may suppose that $f(w)$ is uninitialized, but there exists $P\in \mathcal P\setminus M$ containing $w$ which is initialized. Suppose $P$ has end vertices $w,v$ and $f(w)$ has end vertices $w,u$. Say $Q_P$ has ends $x,y$ (where $x\in A_w$ and $y\in A_v$). \cref{claim:Pneighbours} again yields that there exists $z$ in the interior of $Q_P$ which has a neighbour in $R$. Again take $z$ as close as possible to $x$ in $Q_P$. Given that all active cops are guarding at least one vertex adjacent to $R$ by \cref{claim:sadjacent} and that $A_w$ contains no neighbour of $R$, either all cops of $\mathcal C_{f(w)}$ are inactive, or one of them is guarding a vertex $a\in A_u$ and the others are inactive (in this second case, necessarily $u\neq w$ and so $f(w)$ is not a cycle). Let $C\in \mathcal C_{f(w)}$ be a cop which, depending on the case above, is either inactive or on $a$. \smallskip
	
	    In the first of these two cases, we first send the inactive cop $C$ to sit on $z$ to guard it. Let $A_w'=A_w\cup V(Q_P[x,z])$, $Q_P'=Q_P[z,y]$, $s'(C)=\{z\}$ and $s'(C')=s(C')\setminus \{z\}$, where $C'\in \mathcal C_{P}$ is the cop which was guarding $z$ previously. Then, the game is now in the new state $(\mathcal A',\mathcal Q',R,s')$; this transition has type \ref{item:reduction4}.\smallskip
	
	    In the other case, using \cref{cor:newguardedpath} (note by \cref{claim:copsitting} that $C$ is sitting on $a$) there exists an $a-z$ path $Q$ with internal vertices in $R$ (of which there are at least one), which can, after a finite number of turns (during which $C$ still guards $a$), be guarded by $C$. Let $A_w'=A_w\cup V(Q_P[x,z])$, $Q_P'=Q_P[z,y]$, $Q_{f(w)}'=Q$, $s'(C)=V(Q)$. The game is now in the new state $(\mathcal A',\mathcal Q',R',s')$; this transition has type \ref{item:reduction1}. Note that in this case we initialized $f(w)$, which was necessary as a (or the) cop of $\mathcal C_{f(w)}$ was already busy guarding one vertex. \smallskip
	    
	    In both of these cases, it is important for \ref{item:bagguard} to hold that one of the cops of $C_{f(w)}$ guards $z$, which is now the unique vertex of $A_w'$ adjacent to $R'$. \medskip
	    
	    In all cases, we can reach a strictly smaller game state, which is a contradiction to the minimality of the current game state.
	\end{proof}
	
	\begin{claim}\label{claim:2neighbours}
	    If $C\in \mathcal C$ is such that $s(C)=\{x\}$, where $x\in A_w$ for some $w\in W$, then $x$ has at least 2 neighbours in $R$.
	\end{claim}
	\begin{proof}
	    Suppose to the contrary that $x$ has exactly 1 neighbour $a$ in $R$ (by \cref{claim:sadjacent}, $x$ cannot have no neighbours in $R$). By \cref{claim:copsitting}, $C$ is sitting on $x$. Move $C$ to $a$. Let $A_u'=A_u\cup\{a\}$ and $s'(C)=\{a\}$. The game is now in state $(\mathcal A',\mathcal Q,R',s')$, with transition of type \ref{item:reduction1}, which is a contradiction to the minimality of the current game state. Note that given that $x$ only has $a$ as a neighbour in $R$, $x$ is not a neighbour of $R'$, and so \ref{item:max1} and \ref{item:coboundary} are indeed still respected.
	\end{proof}
	
	\begin{claim}\label{claim:winitialized}
	    Every $w\in W$ is initialized.
	\end{claim}
	\begin{proof}
	    Suppose $w$ is uninitialized. Throughout the proof, let $w,v$ be the ends of $f(w)$. When initializing $w$, recall that by \ref{item:bagguard} we must take care that a cop of $\mathcal C_{f(w)}$ will guard the possible vertex which is adjacent to the robber's territory. We will consider two main cases. \\
	
	    We first consider the case in which $wu\notin M$ for every initialized $u\in W$. Under this assumption, we can initialize $w$ without being concerned with \ref{item:matching} (of course, as long as $w$ is the only vertex we are initializing). There are two subcases to consider. \smallskip
	    
	    The first subcase is if all cops of $\mathcal C_{f(w)}$ are inactive. Let $x\in R$ be arbitrary. Send some $C\in \mathcal C_{f(w)}$ to guard $x$ by sitting on it. Let $A_w'=\{x\}$ and $s'(C)=\{x\}$. It is easy to verify that $(\mathcal A',\mathcal Q,R',s')$ is indeed a new valid state, with transition of type \ref{item:reduction1}. \smallskip
	
	    The second subcase is if not all cops of $\mathcal C_{f(w)}$ are inactive. Given that $A_w=\emptyset$, there can be no paths incident to $A_w$, and so $f(w)$ is necessarily uninitialized. By \ref{item:max1}, \ref{item:noninitsubset}, \cref{claim:sadjacent} and \cref{claim:sdisjoint}, there is only one active cop in $\mathcal C_{f(w)}$, say $C$, which is guarding (and sitting on, by \cref{claim:copsitting}) a vertex $y\in A_v$. Let $x\in R$ be any neighbour of $y$, which must exist by \cref{claim:sadjacent}. It is easy to see that $C$ can guard the path $xy$, for instance by sitting on $y$ and moving to $x$ if the robber goes on $x$. Set $A_w'=\{x\}$, $Q_{f(w)}'=xy$ and $s'(C)=\{x,y\}$. Then $(\mathcal A',\mathcal Q',R',s')$ is a new valid state, with transition of type \ref{item:reduction1}. \medskip
	
	    The other case to consider is if there exists some initialized $u\in W$ such that $wu\in M$. Note that necessarily $w\neq u$ in this case. Given that $M$ is a matching, there is only one such $u$, and so to ensure that \ref{item:matching} is respected after initializing $w$ we only need to consider this $u$. Recall that by definition of $f$, $f(w)\notin M$ and so $f(w)\neq wu$. By \cref{claim:initializedvertex}, there exists $y\in A_u$ such that $y$ has at least one neighbour in $R$. There are once again two main subcases here. \smallskip
	
	    The first subcase now is that all cops of $\mathcal C_{f(w)}$ are inactive, let $C\in \mathcal C_{f(w)}$. Let $x\in R$ be a neighbour of $y$. Send $C$ to guard $x$ by sitting on it. Let $A_w'=\{x\}$, $Q_{wu}'=xy$, and $s'(C)=\{x\}$. Then $(\mathcal A',\mathcal Q',R',s')$ is indeed a new valid state, with transition of type \ref{item:reduction1}. \smallskip
	
	    The other subcase is that there exists $C\in\mathcal C_{f(w)}$ which is active. Given that $A_w$ is uninitialized, $f(w)$ is also uninitialized. As earlier, the cop $C$ must then be guarding (by sitting on) a vertex $x\in A_v$ adjacent to $R$ (and so $w\neq v$), and by \cref{claim:sdisjoint}, any other cop in $\mathcal C_{f(w)}$ is inactive. \medskip
	
	    First suppose $u\neq v$. Apply \cref{cor:newguardedpath} to get a $x-y$ path $Q$ through $R$ of length at least 2, such that $C$ can guard $Q$ after a finite number of turns. Let $z$ be the penultimate vertex of $Q$, i.e. $z$ is the vertex of $Q$ adjacent to $y$. Set $A_w'=\{z\}$, $Q'_{wu}=zy$, $Q'_{f(w)}=Q[x,z]$ and $s(C)=V(Q)\setminus \{y\}$. Then $(\mathcal A',\mathcal Q',R',s')$ is a new valid state, with transition of type \ref{item:reduction1}.  \smallskip
	
	    Now suppose $u=v$. By \ref{item:max1} we have that $x=y$, so we know that $y$ is already guarded by $C$. By \cref{claim:2neighbours}, $y$ has at least two neighbours in $R$, let $a$ be such a such a neighbour. Given that $C\in \mathcal C_{f(w)}$ but is sitting on a vertex of $A_u$, by \ref{item:bagguard} and \cref{claim:sdisjoint} we have that necessarily $f(w)=f(u)$. In particular, $|f^{-1}(f(w))|=2$. Also, given that $wu\in \mathcal P$ already, $f(w)$ cannot be an edge (since $H$ is not a multigraph), and thus has length at least two. Hence, $\ell_{f(w)}\geq 3$. If $\ell_{f(w)}=3$, then $\mathds 1_\ell=1$ and the extra cop $C_\ell$ is present in the game, so let $C'=C_\ell$. If $\ell_{f(w)}\geq 4$, then $|\mathcal C_{f(w)}|\geq 2$ and so let $C'\in \mathcal C_{f(w)}$ which is distinct from $C$ (in particular, $C'$ is inactive). First move $C'$ to $a$. Then, apply \cref{cor:newguardedpath} to get an $a-y$ path $Q$ of length at least two (in particular, not using the edge $ay$) with internal vertices in $R$ which can be guarded by $C'$ after a finite number of turns. Note that is important when applying \cref{cor:newguardedpath} that $C$ keeps guarding $y$ while $C'$ goes to guard $Q$, as otherwise the robber could escape $R$ through $y$. Only once $C'$ is guarding $Q$ can $C$ stop guarding $y$. Let $A_{w}'=\{a\}$, $Q_{wu}'=ay$, $Q_{f(w)}'=Q$, $s'(C')=V(Q)$ and $s'(C)=\emptyset$. If $C'=C_\ell$, we also need to switch the labels of $C$ and $C_\ell$, i.e. $C$ becomes the new extra cop, and $C_\ell$ becomes a cop of $C_{f(w)}$. The game is now in state $(\mathcal A',\mathcal Q',R',s')$, with transition of type \ref{item:reduction1}. \medskip
	    
	    Note that in all of these subcases, the reason no cop is required in $\mathcal C_{wu}$ is because $Q_{wu}'$ is only an edge, hence contains no internal vertex adjacent to the robber's territory. The ends of this path, if they are adjacent to the robber's territory, are protected by the cops designated by $f$ given \ref{item:bagguard}.
	
	    In all cases, we can reach a strictly smaller game state, which is a contradiction to the minimality of the current game state.
	\end{proof}
	
	\begin{claim}\label{claim:Pinitialized}
	    Every $P\in \mathcal P$ is initialized.
	\end{claim}
	\begin{proof}
	    Suppose to the contrary there exists some uninitialized $P\in \mathcal P$. If possible, choose $P$ such that there is $u \in W$ for which $P= f(u)$. By \cref{claim:winitialized}, all vertices in $W$ are initialized. By \ref{item:matching}, $P\notin M$ and so $\mathcal C_P$ is necessarily non-empty.
        
        Suppose $P$ has end vertices $u,v$. There are two main cases to consider: when $u\neq v$ and $u=v$.  \medskip
	
	    First suppose that $u\neq v$. By \cref{claim:initializedvertex}, there exists $x\in A_u$ and $y\in A_v$ adjacent to $R$. As in the previous claims, \ref{item:noninitsubset} implies that any active cop of $\mathcal C_P$ is either sitting on $x$ or $y$, without loss of generality say it is on $x$. If no cop of $\mathcal C_{P}$ is active, first send one inactive cop of $\mathcal C_{P}$ to $x$. In both cases, there is a cop $C\in \mathcal C_{P}$ sitting on $x$. Using \cref{cor:newguardedpath}, there exists at least one $x-y$ path $Q$ with internal vertices in $R$, which can be guarded by $C$ after a finite number of turns (and such that during these turns, $x$ remains guarded by $C$). Define $Q_P'=Q$ and $s'(C)=V(Q)$. Once $C$ is following this new strategy, the game is now in state $(\mathcal A,\mathcal Q',R',s')$, with transition of type \ref{item:reduction1}. \medskip
	
    	We now consider the case $u=v$, so $P$ is an $u$-rooted cycle. By \cref{claim:initializedvertex}, there exists $x\in A_u$ adjacent to $R$. There are two subcases here based on the number of neighbours of $x$ that are in $R$. \smallskip
	    
	    First suppose $x$ has at least two neighbours in $R$, let one of them be $a\in R$. By \ref{item:bagguard}, we know that $x$ is guarded by a cop $C'\in \mathcal C_{f(u)}$. We first want to find a cop (distinct from $C'$) to guard a new path used to initialize $P$. If $P\neq f(u)$, all cops of $\mathcal C_P$ are necessarily inactive by \ref{item:noninitsubset} and \cref{claim:sdisjoint} given that $x$ is already guarded by $C'$, so let $C\in \mathcal C_P$. Suppose now that $P=f(u)$. Recall that we want to find an inactive cop distinct from $C'$. Given that a cycle has length at least 3 we have $\ell_P\geq 3$. If $\ell_P=3$, then $\mathds 1_\ell=1$ and the extra cop $C_\ell$ is present in the game, so let $C=C_\ell$. If $\ell_P\geq 4$, $|\mathcal C_P|\geq 2$ and so let $C\in \mathcal C_P$ which is distinct from $C'$. In both cases, $C$ is inactive and thus available to take on a new strategy. Move $C$ to $a$, and apply \cref{cor:newguardedpath} to get an $a-x$ path $Q$ of length at least two (in particular, not using the edge $ax$) with internal vertices in $R$. This path can be guarded by $C$ after a finite number of turns. Note that it is important when applying \cref{cor:newguardedpath} that $C'$ keeps guarding $x$ while $C$ prepares to guard $P$ (which is why we wanted $C$ to be distinct from $C'$). Set $Q_P'=xa\oplus Q$ and $s'(C)=V(Q)$. In the case where $C=C_\ell$, also set $s'(C')=\emptyset$. This only happens if, in particular, $C'\in \mathcal C_{f(u)}$, and so $C'$ was necessarily sitting on $x$ by \ref{item:noninitsubset}. In this case, we also need to switch the labels of $C'$ and $C_\ell$, that is $C'$ becomes the new extra cop, and $C_\ell$ becomes a cop of $C_{P}$. The game is now in state $(\mathcal A,\mathcal Q',R',s')$, with transition of type \ref{item:reduction1}.\medskip
	
	    Second, suppose $x$ has exactly one neighbour $a$ in $R$. If $P=f(u)$, then by \ref{item:bagguard}, there necessarily exists $C\in \mathcal C_P$ which is guarding $x$. By \ref{item:noninitsubset}, $s(C)=\{x\}$. This contradicts \cref{claim:2neighbours}, so $P\neq f(u)$. In particular, by the same argument in the previous case, no cop of $\mathcal C_P$ is active, so we let $C\in \mathcal C_P$. By our initial choice of $P$, $f(u)$ must be initialized. Let $w$ be the other end of $f(u)$, and let $y$ be the vertex of $A_w$ adjacent to $R$, which exists by \cref{claim:initializedvertex}. We next consider whether the interior of $Q_{f(u)}$ contains vertices adjacent to $R$, and separate this into two subsubcases. \smallskip
	
	    First suppose the interior of $Q_{f(u)}$ contains no vertex adjacent to $R$. Note that by \cref{claim:Pneighbours}, $f(u)$ is not a cycle, so $u\neq w$. Apply \cref{cor:newguardedpath} to get an $x-y$ path $Q$ which goes through $R$ which can be guarded by $C$ after a finite number of turns. The cops of $\mathcal C_{f(u)}$ guarding (the ends of) $Q_{f(u)}$ may now be relieved. Let $Q_{f(u)}'=Q$, $s'(C)=V(Q)$ and $s'(C')=\emptyset$ for every $C'\in \mathcal C_{f(w)}$. Note however that this would no longer respect \ref{item:nonintertwined}, given that $C$ is in $\mathcal C_P$ but is guarding $Q_{f(u)}'$. Hence, switch the roles of $C$ and of one cop $C'\in \mathcal C_{f(w)}$, that is we redefine $\mathcal C_{f(w)}=(\mathcal C_{f(w)}\setminus\{C'\})\cup\{C\}$ and $\mathcal C_{P}=(\mathcal C_{P}\setminus\{C\})\cup\{C'\}$). Then, the game is now in state $(\mathcal A,\mathcal Q',R',s')$, with transition of type \ref{item:reduction1}. \smallskip
	
	    Second, suppose now the interior of $Q_{f(u)}$ contains a vertex adjacent to $R$ (note that here it is possible that $u=w$). Choose $z$ to be such a vertex as close to $x$ as possible, i.e. the interior of $Q_{f(u)}[x,z]$ contains no vertex adjacent to $R$. Apply \cref{cor:newguardedpath} to get an $x-z$ path $Q$ which goes through $R$ which can be guarded by $C$ after a finite number of turns. Define $A_u'=A_u\cup V(Q_{f(u)}[x,z])$,  $Q_{f(u)}'=Q_{f(u)}[z,y]$ and $Q_P'=Q$. Given that $x,z\in A_u'$, $Q_P'$ is indeed an $A_u'-A_u'$ path. Let $s(C)=V(Q)$ and let $s'(C')=s(C')\setminus\{x\}$ for the cop $C'\in \mathcal C_{f(u)}$ which was previously guarding $x$, in order for \ref{item:nonintertwined} to still hold. The game is now in state $(\mathcal A',\mathcal Q',R',s')$, with transition of type \ref{item:reduction1}. Note that given that $x$ only had $a$ as a neighbour in $R$, $a$ is necessarily in $Q$, and so $z$ is the only vertex of $A_u'$ which is potentially adjacent to $R'$, hence \ref{item:max1} still holds.\medskip
	    
	    In all cases, we can reach a strictly smaller game state, which is a contradiction to the minimality of the current game state.
	\end{proof}
	
	\begin{claim}\label{claim:ellPneighbours}
	    For every $P\in \mathcal P$, $\sum_{C\in \mathcal C_P} |s(C)|\geq \ell_P$.
	\end{claim}
	\begin{proof}
        Suppose to the contrary that there exists $P$ such that $\sum_{C\in \mathcal C_P} |s(C)|<\ell_P$.Since $|\mathcal C_P|=\left\lceil\frac{\ell_P}{3}\right\rceil$ by definition, there must be some $C \in \mathcal{C}_P$ with $|s(C)| \leq 2$. If possible, choose this $C$ with $|s(C)| \leq 1$.
        
        By \cref{claim:Pinitialized}, $P$ is initialized, so $s(\mathcal{C}_P)$ is a non-intertwined family of subsets of $V(Q_P)$. Let $u,v$ be the end vertices of $P$, and $x,y$ the end vertices of $Q_P$, where $x\in A_u$ and $y\in A_v$. Notice that when, $\ell_P\in \{0,1,2,4\}$, we have $2\left\lceil\frac{\ell_P}{3}\right\rceil\geq \ell_P$, and so there exists $C\in \mathcal C_P$ such that $s(C)\leq 1$; in this case, such a $C$ would have been chosen above. Hence, if $|s(C)|=2$, we know it must be the case that $\ell_P \notin \{0,1,2,4\}$. \medskip

	    If $|s(C)|=1$, let $z_1$ be the vertex $C$ is currently sitting on. If $|s(C)|=0$, let $z_1$ be the vertex of $Q_P$ which is adjacent to $R$ (such a vertex exists by \cref{claim:Pneighbours}) and closest to $x$ (when traversing $Q_P$ from $x$ to $y$), and then send $C$ to $z_1$. \smallskip
	    
	    If $z_1$ is the only vertex of $Q_P$ adjacent to $R$, by \cref{claim:Pneighbours} $z_1$ is necessarily an internal vertex of $Q_P$. In this case, by \cref{claim:initializedvertex}, $A_v$ must contain a vertex $z_2\neq z_1$ which is adjacent to $R$. Otherwise, let $z_2$ be the first vertex adjacent to $R$ which appears after $z_1$ when traversing $Q_P$ from $x$ to $y$. 
	    
	    By \cref{cor:newguardedpath}, there exists a $z_1-z_2$ path $Q$ of length at least two with internal vertices in $R$ such that $C$ has a strategy to keep guarding $z_1$ and, after a finite number of turns, guard $Q$. If $z_2\in Q_P$, let $Q_P'=Q_P[x,z_1]\oplus Q\oplus Q_P[z_2,y]$. Otherwise we chose $z_2\in A_v$, and so let $Q_P'=Q_P[x,z_1]\oplus Q$. Note that in all cases, the parts of $Q_P$ which are being dropped did not contain any neighbour in $R$, so \ref{item:coboundary} still holds, and $Q_P'$ still has ends in $A_u$ and $A_v$. Let $s(C)=V(Q)$. It is direct that this maintains \ref{item:nonintertwined}. The game is now in state $(\mathcal A',\mathcal Q',R',s')$, with transition of type \ref{item:reduction1}. \medskip
	    
	    Suppose now that $|s(C)|=2$. By the choice of $C$ above, $\ell_P\notin \{0,1,2,4\}$. In particular, $\mathds 1_\ell=1$, and so the cop $C_\ell$ exists and is inactive. Let $z_1,z_2$ be the vertices of $s(C)$, suppose without loss of generality that $z_1$ appears before $z_2$ when traversing $Q_P$ from $x$ to $y$. By \ref{item:nonintertwined} and \cref{claim:sadjacent}, $Q_P[z_1,z_2]$ contains no internal vertex adjacent to $R$ (in the very specific case where $Q_P$ is a cycle with root $x=z_1$, given the somewhat technical definition of non-intertwined for cycles, it is possible that one might need the consider $Q_P$ to be travelled in the opposite direction for this to hold). \smallskip
     
        By \cref{cor:newguardedpath}, there exists a $z_1-z_2$ path, call it $Q$, of length at least two with internal vertices in $R$ such that $C_\ell$ has a strategy to guard $Q$, after a finite amount of turns. Set $Q_P'=Q_P[x,z_1]\oplus Q\oplus Q_P[z_2,y]$, $s(C_\ell)=V(Q)$ and $s(C)=\emptyset$. With $C$ now inactive, we relabel $C$ to be $C_\ell$ and vice versa, in order for \ref{item:nonintertwined} and \ref{item:extracop} to hold. The game is now in state $(\mathcal A',\mathcal Q',R',s')$, with transition of type \ref{item:reduction1}.\medskip
	
	    In all cases, we can reach a strictly smaller game state, which is a contradiction to the minimality of the current game state.
	\end{proof}
	
	\begin{claim}\label{claim:Hminor}
	    $H$ is a minor of $G$.
	\end{claim}
	\begin{proof}
	    We use the model in the current state $(\mathcal A,\mathcal Q,R,s)$ to construct a minor of $H$ in $G$.
	    
	    First, contract all edges in the connected component $R$ and call the resulting vertex $h'$. Any vertex in $V(G)\setminus R$ adjacent to $R$ is now adjacent to $h'$.
	    
	    For every $w\in W$, $A_w$ is non-empty by \cref{claim:winitialized}. By the  definition in \ref{item:defbags}, $G[A_w]$ is connected for every $w \in W$. Hence, we may contract every edge between vertices in $A_w$ to obtain one vertex, which we denote $w'$. Since $A_w$ contains at least one vertex adjacent to $R$ by \cref{claim:initializedvertex}, $h'$ is adjacent to $w'$ in the resulting graph.
	    
	    By \cref{claim:Pinitialized}, every $P\in \mathcal P$ is initialized. For every $P\in\mathcal P$ and every edge $uv\in Q_P$, contract the edge $uv$ if either $u$ or $v$ is not adjacent to $h'$. Let $P'$ be the resulting path (or cycle), which has ends $u'$ and $v'$. With these contractions, every vertex of $P'$ is adjacent to $h'$.
	    
	    By \cref{claim:sdisjoint}, the images of $s$ are disjoint, which contain only vertices adjacent to $R$ (and now $h'$) by \cref{claim:sadjacent}. Let $P\in \mathcal P$ with ends $u,v$. By \ref{item:bagguard}, the end vertices of $Q_P$, say $x,y$, are in one of the sets of $s(\mathcal C_P)$ only if $P=f(u)$ and $P=f(v)$ respectively. Thus, $P'$ contains $\left(\sum_{C\in \mathcal C_P} |s(C)|\right)-|f^{-1}(P)|$ internal vertices. By \cref{claim:ellPneighbours}, $P'$ then contains at least $\ell_P-|f^{-1}(P)|\geq |E(P)|-1$ internal vertices. As the number of edges in a path or a rooted cycle is one more than the number of internal vertices, $P'$ contains at least $|E(P)|$ edges. We may contract further edges of $P'$ in order for $P'$ to contain exactly $|E(P)|$ edges.

	    Mapping $h$ to $h'$, $w$ to $w'$ for every $w\in W$ and $P$ to $P'$ for every $P\in \mathcal P$, we conclude that $H$ is isomorphic to a subgraph of our contracted graph, and so $H$ is a minor of $G$.
	\end{proof}
	
	Given that $G$ is $H$-minor-free, \cref{claim:Hminor} yields the  contradiction. This completes the proof of the theorem.
\end{proof}

\subsection{Key ideas of the proof}\label{subsec:diffs}
We now highlight a few key elements of the proof of \cref{thm:main}.

The method introduced by Andreae in \cite{andreae_pursuit_1986} consists in, as long as the robber is not caught, gradually constructing a minor of $H-h$ by buildings bags corresponding to the vertices of $H-h$ and using path guarding strategies for cops in order to add paths between bags when the corresponding vertices are adjacent in $H-h$. These paths are taken through the robber's territory, gradually reducing its size. Once the minor of $H-h$ is completed, contracting the robber's territory then yields a minor of $H$. As the graph is $H$-minor-free, this process cannot be completed, and hence the cops must eventually capture the robber. Our proof builds on this basic framework in multiple ways.

\begin{enumerate}
	\item In the proof of \cref{thm:andreaemain}, exactly one cop is used to recreate each edge of $H-h$ in the minor by guarding a path between two bags corresponding to adjacent vertices of $H-h$ (which yields the bound $c(G)\leq |E(H-h)|$). In a specific proof sketch for wheel graphs (a cycle plus a universal vertex), Andreae \cite[Theorem 3]{andreae_pursuit_1986} uses the fact that one cop can be used to recreate at least three vertices of the cycle: when a cop is guarding fewer than three (for simplicity, say there are two) vertices adjacent to the robber's territory, the extra cop can relieve this cop by guarding a new path between these two vertices through the robber's territory. Using a more specific assignment of cops, in which now cops are grouped together to guard paths (and rooted cycles) between pairs of ``core'' vertices $W$, the same idea can be used for general graphs.
	\item To go from a model of $H-h$ to a model of $H$, one requirement is that each bag $A_w$ (which will be contracted to give $w$ in the minor) must contain at least one vertex adjacent to $R$ (at least, when $wh\in E(H)$). Furthermore, the existence of at least one such vertex is needed when adding a new $u-v$ path to the model, as it allows us to get a new path between $A_u$ and $A_v$ passing through $R$. In Andreae's proof, when $A_w$ no longer has a neighbour in $R$, it gains one by absorbing parts of one of the paths incident to it, or otherwise it is uninitialized. Using this approach directly with the previous improvement, we would get \cref{cor:simplepaths} below. However this is not optimal, as it requires that the group of cops of any of the paths be large enough to guard not only the required number of neighbours of $R$ internally in the path, but in the ends of the path as well. Hence, another key idea in our proof is that it is in fact possible to designate for each $w$ from which path to absorb vertices to acquire a vertex adjacent to $R$; this is the role played by $f$. If this is not possible, i.e. if $f(w)$ is uninitialized, a neighbour of $R$ will be acquired from another path incident to $A_w$ (if one such path exists), and then we use a cop of $\mathcal C_{f(w)}$ to guard this vertex. This is reflected by \ref{item:bagguard}.
	\item In the last point, what happens if instead of a long path between $u,v\in W$, there is simply an edge, and $uv$ is not in the image of $f$? In some sense, to get the minor we do not need any neighbour of $R$ to be present in the path $Q_{uv}$ which will be contracted to $uv$. However, if we take $Q_{uv}$ to be a path between $A_u$ and $A_v$, a cop is still potentially needed: even though $Q_{uv}$ is not required to contain a neighbour in $R$, it might contain one. However, when we first initialize $A_u$ or $A_v$, we can do so in a way that $Q_{uv}$ is only an edge and thus no cop will need to be assigned to guard it. This is the role the matching $M$ plays. Let us note that this only works when $M$ is a matching; this is a consequence of the fact that we cannot much control the order in which sets $A_w$ are initialized and uninitialized.
	\item When using a group of cops to recreate a path to build the minor of $H-h$, one extra cop is often required. More precisely, we expect every cop to be able to guard at least three vertices adjacent to $R$. However, if one cop is guarding exactly two vertices adjacent to $R$, there is generally no way for that cop to start guarding a new path through $R$ between these vertices without losing control of one of these vertices, and all other cops might also be busy. We can use the extra cop to do so, after which the first cop can be relieved (the cops may then switch their roles); this is what is suggested for wheel graphs in \cite{andreae_pursuit_1986}. In some very specific cases with short paths, we can guarantee that if the cops are not on average guarding at least three neighbours of $R$, then one of these cops is necessarily guarding at most one neighbour of $R$, in which case the extra cop is not required. We will see in the applications in the next section that this difference can be very useful when $H$ is small. We note that the extra cop is also used in very specific technical situations involving edges of the matching or cycles in \cref{claim:winitialized} and \cref{claim:Pinitialized}.
\end{enumerate}

Note that many of the technicalities of the proof concern the interplay between these various improvements. Indeed, we often need to break the proofs of the claims into various cases depending on whether, for instance, the $P\in \mathcal P$ in question is a path or a cycle, is in the image of $f$ or not and is in the matching $M$ or not. These complexities also require a more technical proof statement and system of states and state transitions. Our proof is also quite formal when it comes to path guarding strategies, hence the use of \cref{cor:newguardedpath} and the specific formulation of condition \ref{item:funcs}.

\section{Applications}\label{sec:applications}

In this section, we will see various consequences of our main result \cref{thm:main}.

\subsection{Simplified versions of the main result}
In many cases, one might not need the full flexibility of \cref{thm:main}, which is quite technical. In this section we present some simpler versions of this result. This will also allow us to better isolate the various improvements described in \cref{subsec:diffs}.

Firstly, we have a version of \cref{thm:main} in which the only difference with \cref{thm:andreaemain} is the addition of a matching of ``free'' edges.

\begin{cor}\label{cor:simplematching}
	Let $H$ be a graph, $h\in V(H)$ and $M$ be a matching of $H-h$ such that $H-h-M$ has no isolated vertex. If $G$ is a connected $H$-minor-free-graph, then $c(G)\leq |E(H-h)|-|M|$.
\end{cor}

\begin{proof}
    Let $W=V(H-h)$, let $\mathcal P=E(H-h)$ (considering every edge as a path of length 1) and let $f$ be arbitrary; at least one such function exists since every vertex of $H-h-M$ is not isolated. Then, $(h,W,\mathcal P,M,f)$ is a decomposition of $H$.
    
    For every $P\in \mathcal P$, we have that $|E(P)|=1$ and thus \begin{align*}\ell_P&= \begin{cases}
		0&P\in M\\
		\max(|E(P)|-1+|f^{-1}(P)|,1)&P\notin M
	\end{cases}\\&\leq \begin{cases}
		0&P\in M\\
		2&P\notin M.
	\end{cases}\end{align*}
	
	This implies that $\mathds 1_\ell=0$. Hence, by \cref{thm:main} we have that 
    
	$$c(G)\leq \mathds{1}_{\ell}+\sum_{P\in \mathcal P}\left\lceil\frac{\ell_P}{3}\right\rceil\leq \sum_{e\in E(H-h)\setminus M}\left\lceil\frac{2}{3}\right\rceil+\sum_{e\in M}\left\lceil\frac{0}{3}\right\rceil=|E(H-h)|-|M|.$$
\end{proof}

We might also want a version of \cref{thm:main} in which we use the improvements for long paths in $H-h$ but without some of the technicalities.

\begin{cor}\label{cor:simplepaths}
    Let $H$ be a graph and $h\in V(H)$ be a vertex such that $H-h$ has no isolated vertex. Let $W\subseteq V(H-h)$ be non-empty and let $\mathcal P$ be a collection of pairwise internally vertex-disjoint paths and cycles with end vertices in $W$ such that every edge of $H-h$ is contained in some $P\in\mathcal P$.
	
 If $G$ is a connected $H$-minor-free graph, then
	$$c(G)\leq 1+\sum_{P\in \mathcal P}\left\lceil\frac{|V(P)|}{3}\right\rceil.$$
\end{cor}

\begin{proof}
    Let $f$ be arbitrary; at least one valid choice exists given that no vertex of $W$ is isolated. Then, $(h,W,\mathcal P,\emptyset,f)$ is a decomposition of $H$.
    
    If $P\in \mathcal P$ is a path, $|f^{-1}(P)|\leq 2$, and so $\ell_P\leq |E(P)|+1=|V(P)|$. If $P\in \mathcal P$ is a cycle, then $|f^{-1}(P)|\leq 1$, and so $\ell_P\leq |E(P)|=|V(P)|$. Furthermore, $\mathds{1}_{\ell}\leq 1$.
\end{proof}

\subsection{Recovering Andreae's results}
Here we show that \cref{thm:main} is indeed a generalization of Andreae's results. Firstly, we indeed recover \cref{thm:andreaemain}, which we restate for convenience.

\thmandreaemain*

\begin{proof}
    Apply \cref{cor:simplematching} with $M=\emptyset$.
\end{proof}

Consider the wheel graph $W_t=U(C_t)$ (where $C_t$ is the cycle graph on $t$ vertices). As noted in the introduction, Andreae proved the following. As noted in \cref{subsec:diffs}, the proof of that result partially inspired \cref{thm:main}. We can recover that result.

\begin{thm}
    If $G$ is a connected $W_t$-minor-free graphs ($t\geq 3$), then $c(G)\leq \left\lceil\frac{t}{3}\right\rceil+1$.
\end{thm}

\begin{proof}
    Apply \cref{cor:simplepaths} with $h$ being the universal vertex, $W=\{u\}$ where $u$ is some arbitrary vertex of $W_t-h$ and $\mathcal P$ containing only the cycle $W_t-h$ which we root at $u$.
\end{proof}

Further results of Andreae, for $K_{3,3}$-minor-free graphs and $K_{2,3}$-minor-free graphs, are recovered in the next subsection. We note however that we cannot recover all of Andreae's results for small graphs, in particular the upper bound of 3 on the cop number of connected $K_5$-minor-free graphs. Andreae's method to prove this, although similar to the methods used to prove \cref{thm:andreaemain}, constructs the minor more carefully, in a way which only works for very small graphs. In particular, whereas in the general framework used to prove \cref{thm:andreaemain} and \cref{thm:main} we do not have much control over what $A_w$ (the set of vertices which are going to be contracted to obtain $w$) looks like, in Andreae's proof for $K_5$-minor-free graphs the structure of the model of $H$ is much more rigid; some edges in the minor can be obtained as a consequence of the fact that when building a minor of a small graph, one can keep track of the presence of specific vertices and edges.

\subsection{Complete bipartite graphs}

We can improve the bound from $2t$ to $t$ for $K_{3,t}$-minor-free graphs.
\begin{cor}
	If $G$ is a connected $K_{3,t}$-minor-free graph $(t\geq 2)$, then $c(G)\leq t$.
\end{cor}
\begin{proof}
	Let $h,a,b$ be the vertices in the part of $K_{3,t}$ with 3 vertices. Then, $K_{3,t}-h$ consists of exactly $t$ internally disjoint paths $P_1,\dots,P_t$ of length 2 between $a$ and $b$. Let $W=\{a,b\}$, $\mathcal P=\{P_1,\dots,P_t\}$ and define $f$ by $f(a)=P_1$, $f(b)=P_2$. Then, $(h,W,\mathcal P,\emptyset,f)$ is a decomposition of $H$.
 
    We have that $|E(P_i)|=2$ for $i\in [t]$, $|f^{-1}(P_1)|=|f^{-1}(P_2)|=1$ and $|f^{-1}(P_i)|=0$ for $i\in [t]\setminus\{1,2\}$. Hence $\ell_{P_i}\leq 2$ for every $i\in [t]$, and in particular $\mathds 1_\ell=0$. \cref{thm:main} then yields the result.
\end{proof}
In particular, we recover the bound for $K_{3,3}$-minor-free graphs from \cite{andreae_pursuit_1986} without needing a separate argument.

We can also improve the upper bounds for $K_{2,t}$-minor-free from $t$ to essentially half that.
\begin{cor}\label{cor:k2t}
	If $G$ is a connected $K_{2,t}$-minor-free graph $(t\geq 1)$, then $c(G)\leq \left\lceil\frac{t+1}{2}\right\rceil$.
\end{cor}
\begin{proof}
    First note that it suffices to show the result when $t$ is odd, since $K_{2,{t-1}}$-minor-free graphs are also $K_{2,t}$-minor-free, and in this case $\left\lceil\frac{(t-1)+1}{2}\right\rceil=\left\lceil\frac{t}{2}\right\rceil=\frac{t-1}{2}+1$.
    
	Consider the graph $H=U\left(U\left(\frac{t-1}{2}K_2+K_1\right)\right)$. In other words, if $h$ is one of the universal vertices, $H-h$ is a graph obtained by identifying one vertex of $\frac{t-1}{2}$ triangles and of one edge. In particular, $K_{2,t}$ is a subgraph of $H$, and so it suffices to prove that connected $H$-minor-free graphs have cop number at most $\frac{t-1}{2}+1$. Let $G$ be such a graph.
	
	Let $a$ be the universal vertex of $H-h$. We have that $H-h$ is the union of $\frac{t-1}{2}$ internally-disjoint $a$-rooted cycles of length 3 (write $\mathcal P_{aa}$ for this collection of cycles), and one other edge $ab$. Let $W=\{a,b\}$, let $\mathcal P=\mathcal P_{aa}\cup\{ab\}$ and define $f$ by $f(a)=f(b)=ab$. Then, $(h,W,\mathcal P,\emptyset,f)$ is a decomposition of $H$.
	
	For every cycle (of length 3) $P\in \mathcal P_{aa}$, we have $\ell_P=2$, and $\ell_{ab}=2$. In particular, $\mathds 1_\ell=0$. \cref{thm:main} yields that $c(G)\leq \frac{t-1}{2}+1$ as desired.
\end{proof}

Note that both of these corollaries give us an bound of 2 for $K_{2,3}$-minor-free graphs, which does not follow from \cref{thm:andreaemain} but is a consequence of Andreae's \cite{andreae_pursuit_1986} stronger upper bound of 2 on the cop number of $K_{3,3}^-$-minor-free graphs.

\subsection{Complete graphs}
One of the consequences of \cref{thm:andreaemain} is that if $G$ is a connected $K_t$-minor-free graph for $t\geq 3$, then $c(G)\leq \binom{t-1}{2}=\frac{(t-1)(t-2)}{2}$.

We can improve this result.
\begin{cor}
	If $G$ is a connected $K_t$-minor-free graph $(t\geq 4)$, then $c(G)\leq \left\lfloor\frac{(t-2)^2}{2}\right\rfloor$.
\end{cor}

\begin{proof}
	Let $h$ be an arbitrary vertex of $K_t$, and let $M$ be a maximum matching of $K_t-h\simeq K_{t-1}$, which has size $\left\lfloor\frac{t-1}{2}\right\rfloor$. Note that $K_t-h-M$ contains no isolated vertex since $t\geq 4$. Then, \cref{cor:simplematching} yields that $$c(G)\leq |E(H-h)|-|M|=\binom{t-1}{2}-\left\lfloor\frac{t-1}{2}\right\rfloor=\left\lfloor\frac{(t-2)^2}{2}\right\rfloor.$$
\end{proof}

The cop number of $K_t$-minor-free graphs in particular has received some interest. Andreae \cite{andreae_pursuit_1986} posed as an open problem to find $K_t$-minor-free graphs with large cop number.

Furthermore, Bollobás, Kun and Leader \cite{bollobas_cops_2013} noted the bound on $K_t$-minor-free graphs is related to Meyniel's conjecture. Meyniel's conjecture \cite{frankl_cops_1987} is the most famous and important conjecture on the game of cops and robbers. It states that $c(G)=O(\sqrt n)$ if $G$ is a connected graph on $n$ vertices. A weaker but still open conjecture is the weak or soft Meyniel conjecture, stating that $c(G)=O(n^{1-\delta})$ for some fixed $\delta>0$. Bollobás, Kun and Leader note that if we prove that $K_t$-minor-free graphs have cop number at most $O(t^{2-\varepsilon})$, then weak Meyniel holds for $\delta=\frac{\varepsilon}{4-\varepsilon}$. Briefly, their argument goes as follows. Suppose we wish to bound the cop number of an arbitrary graph $G$ on $n$ vertices. If $G$ has a vertex $u$ of degree $\Omega(n^\delta)$, then place a cop on this vertex and proceed by induction on $G-N[u]$. Otherwise, $G$ has $O(n^{\delta+1})$ edges, and so $G$ cannot contain a complete minor on more than $O(n^{\frac{\delta+1}{2}})$ vertices. We may then apply the bound for graphs forbidding a complete minor to obtain the desired result.

We note that Bollobás, Kun and Leader's argument holds more generally. Suppose $\{G_t\}_{t\geq 1}$ is a family of graphs indexed by $t$ such that $e(t)=|E(G_t)|$ is monotone increasing, and let $f$ be a monotone increasing upper bound on the cop numbers of these graphs, i.e. $c(G_t)\leq f(t)$. If $f(e^{-1}(m))=O(m^{1-\varepsilon})$, then weak Meyniel holds for $\delta=\frac{\varepsilon}{2-\varepsilon}$.

In other words, if we find any class of graphs $G_t$ (not only complete graphs) for which the order of the cop number of $G_t$-minor-free graphs is polynomially smaller than that of the number of edges of $G_t$, one gets an improvement towards Meyniel.

\subsection{Linklessly embeddable graphs}

\begin{figure}
	
	\begin{subfigure}[t]{.24\textwidth}
		\center
			\begin{tikzpicture}[scale=1.4, dot/.style = {circle, fill, minimum
							size=#1, inner sep=0pt, outer sep=0pt}, dot/.default
							= 5pt
				]

				\node [dot] (0) at (2, 2) {};
				
				\node [dot] (1) at (1, 0) {};
				\node [dot] (2) at (2, 0) {};
				\node [dot] (3) at (3, 0) {};
				\node [dot] (4) at (1, 1) {};
				\node [dot] (5) at (2, 1) {};
				\node [dot] (6) at (3, 1) {};
				\node [dot] (7) at (1, 0.5) {};
				\node [dot] (8) at (2, 0.5) {};
				\node [dot] (9) at (3, 0.5) {};

				\draw (0) to (7);
				\draw[bend right=20] (0) to (8);
				\draw (0) to (9);

				\draw (1) to (7);
				\draw (7) to (4);
				\draw (1) to (5);
				\draw[bend right=12] (1) to (6);
				
				\draw (2) to (4);
				\draw (2) to (8);
				\draw (8) to (5);
				\draw (2) to (6);
				
				\draw[bend left=12] (3) to (4);
				\draw (3) to (5);
				\draw (3) to (9);
				\draw (9) to (6);
				
			\end{tikzpicture}
			\caption{$\mathcal P_1$ (Petersen graph)}
		\end{subfigure}
			\begin{subfigure}[t]{.24\textwidth}
	\center
			\begin{tikzpicture}[scale=1.4, dot/.style = {circle, fill, minimum
							size=#1, inner sep=0pt, outer sep=0pt}, dot/.default
							= 5pt
				]

				\node [dot] (0) at (2, 2) {};
				
				\node [dot] (1) at (1, 0) {};
				\node [dot] (2) at (2, 0) {};
				\node [dot] (3) at (3, 0) {};
				\node [dot] (4) at (1, 1) {};
				\node [dot] (5) at (2, 1) {};
				\node [dot] (6) at (3, 1) {};
				
				\node [dot] (7) at (1, 0.5) {};
				\node [dot] (8) at (3, 0.5) {};

				\draw (0) to (7);
				\draw[bend right=20] (0) to (2);

				\draw (0) to (5);
				\draw (0) to (8);

				\draw (1) to (7);
				\draw (4) to (7);
				\draw (1) to (5);
				\draw (1) to (6);
				
				\draw (2) to (4);
				\draw (2) to (5);
				\draw (2) to (6);
				
				\draw (3) to (4);
				\draw (3) to (5);
				\draw (3) to (8);
				\draw (6) to (8);
				
			\end{tikzpicture}
			\caption{$\mathcal P_2$}
		\end{subfigure}
		\begin{subfigure}[t]{.24\textwidth}
		\center
			\begin{tikzpicture}[scale=1.4, dot/.style = {circle, fill, minimum
							size=#1, inner sep=0pt, outer sep=0pt}, dot/.default
							= 5pt
				]

				\node [dot] (0) at (2, 2) {};
				
				\node [dot] (1) at (1, 0) {};
				\node [dot] (2) at (2, 0) {};
				\node [dot] (3) at (3, 0) {};
				\node [dot] (4) at (1, 1) {};
				\node [dot] (5) at (2, 1) {};
				\node [dot] (6) at (3, 1) {};
				
				\node [dot] (7) at (2,0.5) {};

				\draw (0) to (1);
				\draw[bend right=20] (0) to (7);
				\draw (0) to (3);
				\draw (0) to (4);
				\draw (0) to (6);

				\draw (1) to (4);
				\draw (1) to (5);
				\draw[bend right=12] (1) to (6);
				
				\draw (2) to (4);
				\draw (2) to (7);
				\draw (5) to (7);
				\draw (2) to (6);
				
				\draw[bend left=12] (3) to (4);
				\draw (3) to (5);
				\draw (3) to (6);
				
			\end{tikzpicture}
			\caption{$\mathcal P_3$}
		\end{subfigure}
		\begin{subfigure}[t]{.24\textwidth}
			\center
		\begin{tikzpicture}[scale=1.4, dot/.style = {circle, fill, minimum
							size=#1, inner sep=0pt, outer sep=0pt}, dot/.default
							= 5pt
				]

				\node [dot] (0) at (2, 2) {};
				
				\node [dot] (1) at (1, 0) {};
				\node [dot] (2) at (2, 0) {};
				\node [dot] (3) at (3, 0) {};
				\node [dot] (4) at (1, 1) {};
				\node [dot] (5) at (2, 1) {};
				\node [dot] (6) at (3, 1) {};

				\draw (0) to (1);
				\draw[bend right=20] (0) to (2);
				\draw (0) to (3);
				\draw (0) to (4);
				\draw (0) to (5);
				\draw (0) to (6);

				\draw (1) to (4);
				\draw (1) to (5);
				\draw (1) to (6);
				
				\draw (2) to (4);
				\draw (2) to (5);
				\draw (2) to (6);
				
				\draw (3) to (4);
				\draw (3) to (5);
				\draw (3) to (6);
				
			\end{tikzpicture}
			\caption{$\mathcal P_4$}
		\end{subfigure}
		\begin{subfigure}[t]{.3\textwidth}
		\center
			\begin{tikzpicture}[scale=1.4, dot/.style = {circle, fill, minimum
							size=#1, inner sep=0pt, outer sep=0pt}, dot/.default
							= 5pt
				]

				\node [dot] (0) at (2, 2) {};
				
				\node [dot] (1) at (1, 0) {};
				\node [dot] (2) at (2, 0) {};
				\node [dot] (3) at (3, 0) {};
				\node [dot] (4) at (1, 1) {};
				\node [dot] (5) at (2, 1) {};
				\node [dot] (6) at (3, 1) {};

				\draw (0) to (1);

				\draw (0) to (3);
				\draw (0) to (4);
				\draw (0) to (5);
				\draw (0) to (6);

				\draw (1) to (4);
				\draw (1) to (5);
				\draw (1) to (6);
				
				\draw (2) to (4);
				\draw (2) to (5);
				\draw (2) to (6);
				
				\draw (3) to (4);
				\draw (3) to (5);
				\draw (3) to (6);
				
				\draw[bend right=15] (1) to (3);
				
			\end{tikzpicture}
			\caption{$\mathcal P_5$}
		\end{subfigure}
		\begin{subfigure}[t]{.3\textwidth}
		\center
			\begin{tikzpicture}[scale=1.4, dot/.style = {circle, fill, minimum
							size=#1, inner sep=0pt, outer sep=0pt}, dot/.default
							= 5pt
				]

				\node [dot] (0) at (2.5, 2) {};
				
				\node [dot] (1) at (1.5, 0) {};
				\node [dot] (2) at (2.5, 0) {};
				\node [dot] (3) at (3.5, 0) {};
				\node [dot] (4) at (1, 1) {};
				\node [dot] (5) at (2, 1) {};
				\node [dot] (6) at (3, 1) {};
				\node [dot] (7) at (4, 1) {};

				\draw (0) to (4);
				\draw (0) to (5);
				\draw (0) to (6);
				\draw (0) to (7);

				\draw (1) to (4);
				\draw (1) to (5);
				\draw (1) to (6);
				
				\draw (2) to (4);
				\draw (2) to (5);
				\draw (2) to (6);
				
				\draw (3) to (4);
				\draw (3) to (5);
				\draw (3) to (6);
				
				\draw (7) to (1);
				\draw (7) to (2);

			\end{tikzpicture}
			\caption{$\mathcal P_6$ ($K_{4,4}^-$)}
		\end{subfigure}
		\begin{subfigure}[t]{.3\textwidth}
		\center
			\begin{tikzpicture}[scale=1.2, dot/.style = {circle, fill, minimum
							size=#1, inner sep=0pt, outer sep=0pt}, dot/.default
							= 5pt
				]

				\node [dot] (0) at (0, 2) {};
				
				\node [dot] (1) at (0, -1) {};
				\node [dot] (2) at (0.951, -0.309) {};
				\node [dot] (3) at (0.588, 0.809) {};
				\node [dot] (4) at (-0.588, 0.809) {};
				\node [dot] (5) at (-0.951, -0.309) {};

				\draw (1) to (0);
				\draw[bend left=12] (2) to (0);
				\draw (3) to (0);
				\draw (4) to (0);
				\draw[bend right=12] (5) to (0);

				\draw (2) to (1);
				\draw (3) to (1);
				\draw (4) to (1);
				\draw (5) to (1);
				
				\draw (3) to (2);
				\draw (4) to (2);
				\draw (5) to (2);
				
				\draw (4) to (3);
				\draw (5) to (3);
				
				\draw (5) to (4);

			\end{tikzpicture}
			\caption{$\mathcal P_7$ ($K_6$)}
		\end{subfigure}
	\caption[]{Petersen family.\footnotemark}\label{fig:petersen}
\end{figure}
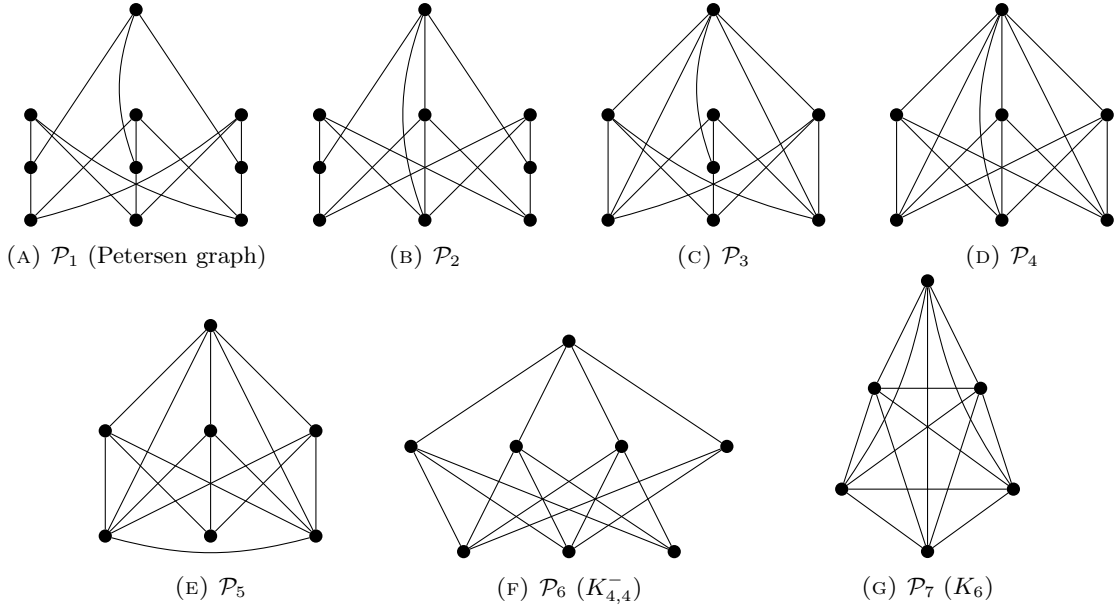

\footnotetext{Drawings based on \cite{robertson_survey_1991}.}

A \emph{linkless embedding} of a graph is an embedding of the graph into $\mathbb{R}^3$ such that every pair of two disjoint cycles forms a trivial link (i.e., they do not pass through one another). A graph that has a linkless embedding is called \emph{linklessly embeddable}. Robertson, Seymour, and Thomas \cite{robertson_linkless_1993} showed that the linklessly embeddable graphs are exactly the graphs excluding the \emph{Petersen family} (see \cref{fig:petersen}) as minors. The Petersen family contains seven graphs that are all $\Delta-Y$ equivalent (i.e. can be obtained by replacing an induced claw by a triangle) to $K_6$, which notably includes $K_{4,4}^-$ and the Petersen graph.

Given the various topological results on the game of cops and robbers discussed in the introduction, one might then be interested in determining the maximum cop number of linklessly embeddable graphs. It follows from \cref{thm:andreaemain} that for any linklessly embeddable graph, $c(G) \le 9$: take $H=\mathcal P_4$ and let $h$ be the degree $6$ vertex ($H-h$ is then $K_{3,3}$).

Using our main result, we are able to improve this upper bound.

\begin{cor}
	If $G$ is a $\mathcal P_i$-minor-free graph $(i\in [4])$, then $c(G)\leq 6$. In particular, if $G$ is a connected linklessly embeddable graph, then $c(G)\leq 6$.
\end{cor}
\begin{proof}
	Let $h$ be the top vertex of $\mathcal P_i$ in the drawings in \cref{fig:petersen}. Then, we can represent $\mathcal P_i-h$ as follows.  Let $W=\{a_1,a_2,a_3,b_1,b_2,b_3\}$ be vertices of $\mathcal P_i-h$ such that $a_ib_j$ is an edge for every distinct $i, j\in [3]$, and for every $i\in [3]$ there is either an edge or a path of length 2 between $a_i$ and $b_i$. Let $\mathcal P$ be the collection of these paths of length 1 or 2. Let $M=\{a_1b_2,a_2b_3,a_3b_1\}$. For $i\in [3]$, let $f(a_i)=f(b_{i+2})=a_ib_{i+2}$ (with indices modulo 3). With these choices, $(h,W,\mathcal P,M,f)$ is a decomposition of $h$. 
 
    We then have that $\ell_P\leq 2$ for every $P\in \mathcal P\setminus M$ (of which there are 6), since either $P$ is an edge which is the $f$-image of 2 vertices of $W$ (hence $\ell_P=2$), or $P$ is a path with 2 edges but is not in the image of $f$ (hence $\ell_P=1$). In particular, $\mathds 1_\ell=0$. \cref{thm:main} then yields that $c(G)\leq 6$.
\end{proof}

As for lower bounds, we were unable to find any linklessly embeddable graph with cop number at least 4 (planar graphs being linklessly embeddable, the dodecahedral graph is an example with cop number 3 \cite{aigner_game_1984}). It hence remains open to determine what the maximum cop number of linklessly embeddable graphs is.

One might also be interested in relating the cop number to the {\it Colin de Verdi\`ere} spectral graph parameter $\mu(G)$. While we omit the formal definition here, Colin de Verdi\`ere showed that $\mu(G) \le 1$ if and only if $G$ is disjoint union of paths,  $\mu(G) \le 2$ if and only if $G$ is outerplanar, and $\mu(G) \le 3$ if and only if $G$ is planar \cite{colin_de_verdiere_sur_1990}. This pattern was extended by Van der Holst, Lov{\'a}sz and Schrijver who showed $\mu(G) \le 4$ if and only if $G$ is linklessly embeddable \cite{van_der_holst_colin_1999}.

For the first three of these classes, it turns out the upper bound on the cop number is the same as the upper bound on the Colin de Verdière invariant, and that the examples for which the cop number bound is tight are also tight for the Colin de Verdière bounds. It is trivial that paths have cop number $1$. Furthermore, if $G$ is a connected outerplanar graph, i.e. a graph which can be embedded in the plane without edge crossings and such that all vertices are on the outer face, then $c(G)\leq 2$. This was originally stated and proved by Clarke \cite{clarke_constrained_2002}, however it was also a consequence of Andreae's bound for $K_{2,3}$-minor-free graphs and $K_4$-minor-free graphs, as it is well-known that the outerplanar graphs are exactly the $\{K_{2,3},K_4\}$-minor-free graphs. Any outerplanar graph of cop number 2 (for example, a cycle of length at least 4), is necessarily not a path and thus has Colin de Verdière number 2. Finally, as mentioned earlier, Fromme and Aigner \cite{aigner_game_1984} have proved that any connected planar graph $G$ has $c(G)\le 3$. Any planar graph of cop number 3 (for example, the dodecahedral graph) is necessarily not outerplanar, and thus must have Colin de Verdière number 3.

Hence, one might wonder whether this pattern continues to linklessly embeddable graphs with an upper bound of 4 for the cop number in this class, and more generally whether $c(G)\leq \mu(G)$ for all connected graphs $G$.

\subsection{Greater improvement factor}\label{subsec:factor}
We have seen earlier that we can improve the bound on $H$-minor-free graphs by a factor of 3 (relative to the number of edges of $H-h$) when $H-h$ can be obtained by subdividing the edges of another graph many times, for instance when applying \cref{cor:simplepaths} to the wheel graph. In fact, in some cases we can essentially get an improvement factor of 4. Let us see an example.

Define $H_t$ to be the graph formed by identifying the end vertices of $m$ copies of a five vertex path, as shown in \cref{fig:factor4}. Recall that $U(H_t)$ is the graph $H_t$ with an additional universal vertex. \cref{thm:andreaemain} shows that if $G$ is a connected $U(H_t)$-minor-free graph, then $c(G)\leq 4m$. We can improve this result by almost a factor of 4.

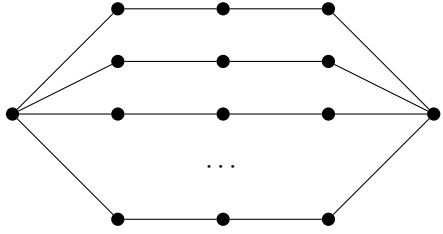
\begin{figure}[H]
\begin{tikzpicture}[scale=1.4, dot/.style = {circle, fill, minimum
							size=#1, inner sep=0pt, outer sep=0pt}, dot/.default
							= 5pt
				]

				\node [dot] (A) at (0, 0) {};
				\node [dot] (B) at (4, 0) {};
				
				\node [dot] (11) at (1, 0) {};
				\node [dot] (12) at (2, 0) {};
				\node [dot] (13) at (3, 0) {};
				
				\node [dot] (21) at (1, 0.5) {};
				\node [dot] (22) at (2, 0.5) {};
				\node [dot] (23) at (3, 0.5) {};

				\node [dot] (31) at (1, 1) {};
				\node [dot] (32) at (2, 1) {};
				\node [dot] (33) at (3, 1) {};
				
				\node [dot] (41) at (1, -1) {};
				\node [dot] (42) at (2, -1) {};
				\node [dot] (43) at (3, -1) {};
				
				\node (dots) at (2, -0.5) {$\cdots$};
				
				\draw (A) to (11);
				\draw (11) to (12);
				\draw (12) to (13);
				\draw (13) to (B);
				
				\draw (A) to (21);
				\draw (21) to (22);
				\draw (22) to (23);
				\draw (23) to (B);
				
				\draw (A) to (31);
				\draw (31) to (32);
				\draw (32) to (33);
				\draw (33) to (B);
				
				\draw (A) to (41);
				\draw (41) to (42);
				\draw (42) to (43);
				\draw (43) to (B);
				
			\end{tikzpicture}
\caption{The graphs $H_t$.}\label{fig:factor4}
\end{figure}

\begin{cor}
	If $G$ is a connected $U(H_t)$-minor-free graph $(t\geq 1)$, then $c(G)\leq t+2$.
\end{cor}
\begin{proof}
    Let $h$ be the universal vertex of $U(H_t)$ and let $a,b$ be the two end vertices of $H_t$ (those obtained by identification). Let $W=\{a,b\}$ and let $\mathcal P=\{P_1,\dots,P_t\}$ be the collection of $t$ internally disjoint $a-b$ paths of length 4 of $H_t$. Finally, define $f$ by $f(a)=f(b)=P_1$. Then, $(h,W,\mathcal P,\emptyset,f)$ is a decomposition of $H$.

    We have that $\ell_{P_1}=5$ and $\ell_{P_i}=3$ for $i\in [t]\setminus\{1\}$. With these values, $\mathds 1_\ell=1$. \cref{thm:main} yields that
	$$c(G)\leq \mathds{1}_{\ell}+\left\lceil\frac{\ell_{P_1}}{3}\right\rceil+\sum_{i=2}^t\left\lceil\frac{\ell_{P_i}}{3}\right\rceil=1+2+(t-1)=t+2.$$
\end{proof}

\section{Future directions}

In \cref{subsec:factor}, we have seen that our results allow us to, in some cases, obtain an improvement of factor 4 over the previous results. There still appears to be a lot of work to be done further optimizing the upper bounds on the cop number when forbidding an minor, both for general classes of graphs and specific graphs.

It would be interesting to get a better upper bound on the cop number when forbidding multiple minors, especially when they are very similar (for instance, for linklessly embeddable graphs). This might in particular yield interesting results for various topological classes of graphs, where the obstruction set usually contains a large number of graphs.

Finding lower bounds, i.e. constructing graphs with some forbidden minor but relatively high cop number, also appears difficult.

\section*{Acknowledgements}
We thank the reviewers for their helpful comments.

\bibliographystyle{abbrvurl}
\bibliography{refs}

\end{document}